\documentclass[11pt]{amsart}

\usepackage{latexsym}
\usepackage{amssymb}
\usepackage{amsmath}
\usepackage{color}
\usepackage{bbm}

\usepackage{graphicx}
\usepackage{tcolorbox}
\usepackage{tabularx}

\usepackage{tikz}
\usepackage{xparse}% So that we can have two optional parameters

\NewDocumentCommand\DownArrow{O{2.0ex} O{black}}{%
   \mathrel{\tikz[baseline] \draw [<-, line width=0.5pt, #2] (0,0) -- ++(0,#1);}
}
\NewDocumentCommand\UpArrow{O{2.0ex} O{black}}{%
   \mathrel{\tikz[baseline] \draw [->, line width=0.5pt, #2] (0,0) -- ++(0,#1);}
}

\newtheorem{theorem}{Theorem}[section]
\newtheorem{lemma}[theorem]{Lemma}
\newtheorem{proposition}[theorem]{Proposition}
\newtheorem{corollary}[theorem]{Corollary}
\newtheorem{definition}[theorem]{Definition}

\newtheorem{remark}[theorem]{Remark}

\newcommand\supp{\mathop{\rm supp}}

\newcommand{\cl}[1]{\mathcal{#1}}
\newcommand{\bb}[1]{\mathbb{#1}}

\newcommand\Tr{\mathop{\rm Tr}}

\begin{document}

\title[Homomorphisms of quantum hypergraphs]
{Homomorphisms of quantum hypergraphs}

%\author[I. G. T.]{I. G. T.}
%\address{Pure Mathematics Research Centre,
%  Queen's University Belfast, Belfast BT7 1NN, United Kingdom}
%\email{i.todorov@qub.ac.uk}

\author[G. Hoefer]{Gage Hoefer}
\address{Department of Mathematical Sciences\\ University of Delaware\\ 501 Ewing Hall\\ Newark\\ DE 19716\\ USA}\email{ghoefer@udel.edu}

\author[I.G. Todorov]{Ivan G. Todorov}
\address{Department of Mathematical Sciences\\ University of Delaware\\ 501 Ewing Hall\\ Newark\\ DE 19716\\ USA} 
\email{todorov@udel.edu}

\date{10 November 2023}


\begin{abstract} 
We introduce quantum homomorphisms between quantum hypergraphs
through the existence of perfect strategies for quantum non-local games, 
canonically associated with the quantum hypergraphs. 
We show that the relation of homomorphism of a given type satisfies natural analogues of the 
properties of a pre-order. We show that quantum hypergraph homomorphisms of 
local type are closely related, and in some cases identical, to the TRO equivalence 
of finite dimensionally acting operator spaces, canonically associated with the hypergraphs. 
\end{abstract}

\maketitle

\tableofcontents

%%%%%%%%%%%%%%%%%%%%%%%%%%%%%%%%%%%%%%%%%%%%%%%%%
%%%%%%%%%%%%%%%%%%%%%%%%%%%%%%%%%%%%%%%%%%%%%%%%%

\section{Introduction}\label{s_intro}

The quantum chromatic number of a graph was introduced in \cite{cmnsw}, 
initiating a wealth of interactions
between graph theory and quantum information theory.
In an influential paper in the field of non-local games \cite{mr}, 
Man\v cinska and Roberson extended the setup of \cite{cmnsw} and defined 
quantum versions of graph homomorphisms. 
The motivation behind these developments is rooted 
in questions about non-locality in quantum mechanics. 
In fact, different models of quantum mechanics lead to 
a hierarchy of no-signalling correlation types between the players of a 
non-local game which, in their own right, can be used to 
witness the distinctions of the quantum mechanical models through 
non-local game theory. 
The main classes used are those of local ($\cl C_{\rm loc}$), 
quantum ($\cl C_{\rm q}$), approximately quantum ($\cl C_{\rm qa}$) and quantum 
commuting ($\cl C_{\rm qc}$) correlations. 
These approaches have led to a proof of 
Bell's Theorem ($\cl C_{\rm loc} \subsetneq \cl C_{\rm q}$) \cite{bell} via the CHSH game \cite{chsh} and, recently, to 
an answer to the Tsirelson, and therefore Connes Embedding, Problem 
in quantum physics and operator algebra theory, respectively ($\cl C_{\rm qa} \subsetneq \cl C_{\rm qc}$) \cite{jnvwy}.

Motivated by the use of graphs in zero-error information theory which was pioneered by Shannon \cite{shannon}, 
a quantum version of graphs was proposed in \cite{dsw}; non-commutative graphs
introduced therein are linear subspaces of the space $M_n$ of all $n$ by $n$ complex matrices, 
which contain the identity matrix and are invariant under the adjoint operation, that is, they are 
finite dimensionally acting operator systems \cite{Pa}. 
Linear algebraic methods have thus been paramount in the study of 
these objects, leading to significant progress in quantised combinatorics (see e.g. \cite{btw, btw2})
and quantum graph theory through the lens of non-local games \cite{bhtl, tl}.

A quantisation of homomorphisms between (classical) hypergraphs, 
from the perspective of non-local game theory, was proposed by the authors in \cite{ght}. 
Hypergraph homomorphisms and isomorphisms, assisted by a given no-signalling correlation class, 
were introduced, characterised in operator algebraic terms, and applications to 
one-shot and asymptotic values of non-local games were provided. 
In the present paper, we introduce a quantum version of hypergraphs, 
and of homomorphisms between quantum hypergraphs, assisted by a given 
quantum no-signalling correlation class. 
This is achieved by requiring that a quantum non-local game, canonically associated with the 
given quantum hypergraphs, possesses a perfect strategy of a given type. 
We use the hierarchy of quantum no-signalling correlations exhibited in \cite{tl}, 
and the quantum channel simulation paradigm
%At the heart of our approach in \cite{ght} was the simulation paradigm for classical information channels 
introduced in \cite{dw}. 
Similarly to the case of classical no-signalling correlations, quantum ones 
admit a natural correlation chain 
$\cl Q_{\rm loc}\subseteq \cl Q_{\rm q} \subseteq \cl Q_{\rm qa}  \subseteq \cl Q_{\rm qc}\subseteq \cl Q_{\rm ns}$.
We show that there is a natural operation of composition of correlations, coherent with 
channel simulation, that preserves each one of the aforementioned classes.
This allows us to show that quantum hypergraph homomorphisms of any given type 
satisfy natural analogues of the properties of a pre-order. 
Finally, we characterise the existence of a quantum hypergraph homomorphism of local type, 
showing that it is closely related, and under certain conditions identical, to 
TRO equivalence between operator spaces, canonically associated by the quantum hypergraphs. 
We note that TRO equivalence \cite{elef} is a concrete operator version of one of the main 
equivalences of Morita type between operator spaces (see \cite{ept}). 
Our results point to an operational, quantum information, route towards Morita equivalence in the 
operator space category. 

The definition of quantum hypergraph homomorphisms relies on the intermediate notion of 
quasi-homomorphisms, pertinent to zero-error information questions (see \cite{ght}). 
We exhibit examples of separation between local and quantum quasi-homomorphisms, as well as between 
quantum and no-signalling ones; the same question in the case of 
homomorphisms remains, however, open.

The paper is organised as follows. 
Section \ref{s_prel} contains preliminary notions and observations.
In Section \ref{s_sqc} we discuss quantum channel simulation and show 
that simulators can be composed with preservation of type, a result we believe may be of interest in its own right.
Section \ref{s_quantize} is dedicated to the main properties of quantum hypergraph homomorphisms and 
their relation with quantum homomorphisms of classical hypergraphs. 
Finally, Section \ref{s_Morita} contains the aforementioned connection between local homomorphisms 
and TRO equivalence of subspaces of matrix algebras.

\subsection*{Acknowledgements}
The research on the topic of the paper was supported by NSF grants CCF-2115071 and DMS-2154459.

%%%%%%%%%%%%%%%%%%%%%%%%%%%%%%%%%%%%%%%%%%%%%%%%%
%%%%%%%%%%%%%%%%%%%%%%%%%%%%%%%%%%%%%%%%%%%%%%%%%

\section{Preliminaries}\label{s_prel}

For a finite set $X$, we let $\bb{C}^X = \oplus_{x\in X}\bb{C}$ and write $(e_x)_{x\in X}$ for the canonical 
orthonormal basis of $\bb{C}^X$. 
We denote by 
$M_X$ the algebra of all complex matrices over $X\times X$, 
and by $\cl D_X$ its subalgebra of all diagonal matrices. 
We write $\epsilon_{x,x'}$, $x,x'\in X$, for the canonical matrix units in $M_X$, denote by 
$\Tr$ the trace functional on $M_X$, and set $\langle S,T\rangle = \Tr(ST^{\rm t})$ (here $T^{\rm t}$ denotes the transpose of $T$ in the canonical basis). For a Hilbert space $H$, let $\cl B(H)$ be 
the C*-algebra of all bounded linear operators on $H$, and denote by $I_{H}$ the identity operator on $H$. 
An \emph{operator system} in $\cl B(H)$ is a selfadjoint linear subspace $\cl S$ of $\cl B(H)$ such that 
$I_H\in\cl S$. 
We write $\overline{H}$ for the dual Banach space of $H$; 
by virtue of the Riesz Representation Theorem, there exists a conjugate linear isometry 
$\partial : H\rightarrow \overline{H}$, such that $\partial(\xi)(\eta) = \langle\eta, \xi\rangle, \xi, \eta \in H$. 
In what follows, we will write $\overline{\bb{C}}^{X} = \overline{\bb{C}^{X}}$. 
We set $\overline{\xi} = \partial(\xi)$. 
Given a linear operator $A: H\rightarrow K$, let $\overline{A}: \overline{K}\rightarrow \overline{H}$ be its (Banach space) dual operator.
We note the formula $\overline{A}(\overline{\xi}) = \overline{A^*\xi}$.
%and the canonical identification $\overline{H\otimes K} = \overline{H}\otimes \overline{K}$. 

We denote by $\cl V\otimes \cl W$ the algebraic tensor product of vector spaces $\cl V$ and $\cl W$, 
except when $\cl V$ and $\cl W$ are Hilbert spaces, in which case the notation is used for their Hilbertian 
tensor product. 
If $\cl S\subseteq \cl B(H)$ and $\cl T\subseteq \cl B(K)$ are operator systems (where $H$ and $K$ are 
Hilbert spaces), we let $\cl S\otimes_{\min}\cl T$ be the \emph{minimal tensor product} of $\cl S$ and $\cl T$
\cite{kptt}- that is, the operator system arising from the inclusion of $\cl S\otimes\cl T$ into $\cl B(H\otimes K)$. 
Given sets $X_i$, $i = 1,\dots,n$, we abbreviate 
$X_1\cdots X_n = X_1 \times \cdots \times X_n$, and 
write $M_{X_1\cdots X_n} = \otimes_{i=1}^n M_{X_i}$ and $\cl D_{X_1\cdots X_n} = \otimes_{i=1}^n \cl D_{X_i}$. 
We let $L_{\omega} : M_{X_1X_2}\to M_{X_1}$ be 
the slice map with respect to a given element $\omega\in M_{X_2}\equiv \cl L(\bb{C}^{X_2})^*$; thus, 
$$L_{\omega}(T_1\otimes T_2) = \langle \omega,T_2\rangle T_1, \ \ T_i\in M_{X_i}, i = 1,2.$$
The partial trace ${\rm Tr}_{X_2} : M_{X_1X_2}\to M_{X_1}$ is the slice map with respect to the identity operator $I_{X_2}$
of $M_{X_2}$. Abusing notation slightly, we also let $L_{u} : H\otimes K\rightarrow K$ be the slice map with respect to a given element $\overline{u} \in \overline{H}$.

Given finite dimensional Hilbert spaces $H$ and $K$, 
and vectors $\xi \in H$ and $\eta \in K$, let $\eta\xi^{*}: H \rightarrow K$ be the 
rank one operator, given by
$(\eta\xi^{*})(\xi') = \langle \xi', \xi\rangle \eta$. 
Let $\theta: \overline{H}\otimes K\rightarrow \cl{L}(H, K)$ be the linear isomorphism given by
$$ \theta(\overline{\xi}\otimes \eta) = \eta\xi^{*}, \;\;\;\; \xi \in H, \eta \in K.$$
The identities in the next lemma will be used throughout; we include the proof 
for the convenience of the reader.

\begin{lemma}\label{l_bars}
Let $X$ and $Y$ be finite sets, 
$u \in \bb{C}^{X_{1}}\otimes \overline{\bb{C}}^{Y_{1}}$, 
$A\in \cl L(\bb{C}^{X_2},\bb{C}^{X_1})$ and $B\in \cl L(\bb{C}^{Y_1},\bb{C}^{Y_2})$. 
Then
$\theta(\overline{u}) = \overline{\theta(u)}^{*}$ and 
$\theta((\overline{A}\otimes B)\overline{u}) = B\theta(\overline{u})A$.
\end{lemma}

\begin{proof}
Let $u = \sum_{i=1}^m \lambda_i\xi_{i}\otimes \overline{\eta}_i \in \bb{C}^{X_{1}}\otimes \overline{\bb{C}}^{Y_{1}}$.
Then $\overline{u} = \sum_{i=1}^m\overline{\lambda}_{i}\overline{\xi}_{i}\otimes \eta_{i}$, and so 
$$\overline{\theta(u)}^{*}
= 
\overline{\sum_{i=1}^m \lambda_{i}\overline{\eta}_{i}\overline{\xi}_{i}^{*}}^*
= 
\left(\sum_{i=1}^m \lambda_{i}\xi_{i}\eta_{i}^*\right)^{*}\\
= 
\sum_{i=1}^m \overline{\lambda_{i}}\eta_{i}\xi_{i}^* = \theta(\overline{u}).$$

For the second identity, note that, if $\gamma \in \bb{C}^{X_{2}}$, then 
\begin{eqnarray*}
\theta((\overline{A}\otimes B)\overline{u})(\gamma)
& = &
\sum_{i=1}^m\overline{\lambda}_{i}\theta((\overline{A}\otimes B)(\overline{\xi}_{i}\otimes \eta_{i}))(\gamma)\\
& = &
\sum_{i=1}^m\overline{\lambda}_{i}\theta(\overline{A^*\xi}_{i}\otimes B\eta_{i})(\gamma)
= 
\sum_{i=1}^m\overline{\lambda}_{i} (B\eta_{i})(A^*\xi_{i})^*(\gamma)\\
& = & 
\sum_{i=1}^m\overline{\lambda}_{i} \langle \gamma, A^*\xi_{i}\rangle B\eta_{i}
= 
\sum_{i=1}^m\overline{\lambda}_{i} \langle A\gamma,\xi_{i}\rangle B\eta_{i}\\
& = & 
\sum_{i=1}^m\overline{\lambda}_{i} B(\eta_i \xi_i^*)A(\gamma)
=
B\theta(\overline{u})A(\gamma).
\end{eqnarray*}
\end{proof}

In the remainder of this section, we recall the basic types of quantum and classical no-signalling 
correlations that will be used in the sequel. 
Let $X, Y, A$ and $B$ be finite sets. 
A \textit{quantum no-signalling (QNS) correlation} \cite{dw} is a quantum channel 
$\Gamma: M_{XY}\rightarrow M_{AB}$ such that
\begin{gather}
	{\rm Tr}_{A}\Gamma(\rho_{X}\otimes \rho_{Y}) = 0 \text{ whenever } \rho_{X} \in M_{X} \text{ and } {\rm Tr}(\rho_{X}) = 0,
\end{gather}
\noindent and
\begin{gather}
	{\rm Tr}_{B}\Gamma(\rho_{X}\otimes \rho_{Y}) = 0 \text{ whenever } \rho_{Y} \in M_{Y} \text{ and } {\rm Tr}(\rho_{Y}) = 0.
\end{gather}
We set 
$$\Gamma(aa',bb'|xx',yy') = \langle \Gamma(\epsilon_{x,x'} \otimes \epsilon_{y,y'}),\epsilon_{a,a'} \otimes \epsilon_{b,b'}
\rangle;$$
thus, 
$(\Gamma(aa',bb'|xx',yy'))_{x,x',a,a'}^{y,y',b,b'}$ is the Choi matrix of $\Gamma$ (see e.g. \cite{Pa}).
A \emph{stochastic operator matrix} acting on a Hilbert space $H$ 
is a positive block operator matrix $E = (E_{x, x', a, a'})_{x, x', a, a'} \in M_{XA}(\cl{B}(H))$ such that 
${\rm Tr}_{A}E = I_{X} \otimes I_{H}$.
A QNS correlation $\Gamma: M_{XY}\rightarrow M_{AB}$ is called 
\emph{quantum commuting} if there exists a Hilbert space $H$, a unit vector $\xi \in H$ and stochastic operator matrices $\tilde{E} = (E_{x, x', a, a'})_{x, x', a, a'}$ and $\tilde{F} = (F_{y, y', b, b'})_{y, y', b, b'}$ on $H$ such that 
$$E_{x, x', a, a'}F_{y, y', b, b'} = F_{y, y', b, b'}E_{x, x', a, a'}$$
for all $x, x' \in X, y, y' \in Y, a, a' \in A, b, b' \in B$, and
\begin{equation}\label{eq_EFp}
\Gamma(\epsilon_{x,x'} \otimes \epsilon_{y,y'}) = \sum_{a,a'\in A} \sum_{b,b'\in B}
\left\langle E_{x,x',a,a'}F_{y,y',b,b'}\xi,\xi \right\rangle \epsilon_{a,a'} \otimes \epsilon_{b,b'}, 
\end{equation}
for all $x,x' \in X$ and all $y,y' \in Y$.
\emph{Quantum} QNS correlations are defined as in (\ref{eq_EFp}), but 
requiring that $H$ has the form $H_A\otimes H_B$, for some finite dimensional Hilbert spaces $H_A$ and $H_B$, and 
$E_{x,x',a,a'} = \tilde{E}_{x,x',a,a'} \otimes I_B$ and 
$F_{y,y',b,b'} = I_A \otimes \tilde{F}_{y,y',b,b'}$, for some stochastic operator matrices 
$(\tilde{E}_{x,x',a,a'})$ and 
$(\tilde{F}_{y,y',b,b'})$, acting on $H_A$ and $H_B$, respectively. 
\emph{Approximately quantum} QNS correlations are the 
limits of quantum QNS correlations, while \emph{local} QNS correlations are 
the convex combinations of the form 
$\Gamma = \sum_{i=1}^k \lambda_i \Phi_i \otimes \Psi_i$, 
where $\Phi_i : M_X\to M_A$ and $\Psi_i : M_Y\to M_B$ are quantum channels, $i = 1,\dots,k$.
We write $\cl Q_{\rm qc}$ (resp. $\cl Q_{\rm qa}$, $\cl Q_{\rm q}$, $\cl Q_{\rm loc}$) for the (convex) set of all quantum commuting (resp. approximately quantum, quantum, local) QNS correlations, and note the (strict, see \cite{tl})
inclusions 
\begin{equation}\label{eq_Qinc}
\cl Q_{\rm loc}\subseteq \cl Q_{\rm q}\subseteq \cl Q_{\rm qa}\subseteq \cl Q_{\rm qc}\subseteq \cl Q_{\rm ns}.
\end{equation}
Let
$$	\cl{L}_{X, A} = \bigg\{(\lambda_{x, x', a, a'}) \in M_{XA} :  \exists \;c \in \bb{C} \;\text{s.t.} \;
	\sum_{a\in A}\lambda_{x, x', a, a} = \delta_{x, x'}c, \; x, x' \in X\bigg\},$$
and consider it as an operator subsystem of $M_{XA}$. 
By \cite[Proposition 5.5, Theorem 6.2]{tl}, the elements $\Gamma$ of $\cl Q_{\rm ns}$ 
correspond canonically to elements of the tensor product  $\cl{L}_{X, A}\otimes_{\min} \cl L_{Y,B}$
(viewed as an operator subsystem of $M_{XA}\otimes M_{YB}$). 

If $X$ and $Y$ are finite sets, we write $\Delta_X : M_X\to \cl D_X$ for the diagonal expectation. 
Given a classical information channel $\cl N : \cl D_X\to \cl D_Y$, we let 
$\Gamma_{\cl N} = \cl N\circ \Delta_X$; thus,
$\Gamma_{\cl N} : M_X\to M_Y$ is a quantum channel. Conversely, given a quantum channel 
$\Gamma : M_X\to M_Y$, let $\cl N_{\Gamma} = \Delta_Y\circ \Gamma|_{\cl D_X}$; thus, 
$\cl N_{\Gamma} : \cl D_X\to \cl D_Y$ is a classical information channel. 
For a classical channel $\cl N : \cl D_X\to \cl D_Y$, we write 
$\cl N(y|x) = \langle \cl N(\epsilon_{x,x}),\epsilon_{y,y}\rangle$ and set 
$${\rm supp}(\cl N) = \{(x,y) :  \cl N(y|x) \neq 0\}$$
to be the \emph{support} of $\cl N$. 

%Let $X_i$ and $Y_i$ be finite sets. 
A \emph{classical no-signalling correlation} over $(X,Y,A,B)$ 
is a classical information channel 
$\cl N : \cl D_{XY}\to \cl D_{AB}$ such that $\Gamma_{\cl N}$ is a QNS correlation. 
If ${\rm t}\in \{{\rm loc, q, qa, qc, ns}\}$, we let $\cl C_{\rm t}$ be the collection of all 
classical NS correlations $\cl N$ for which $\Gamma_{\cl N}\in \cl Q_{\rm t}$.

%%%%%%%%%%%%%%%%%%%%%%%%%%%%%%%%%%%%%%%%%%%%%%%%%
%%%%%%%%%%%%%%%%%%%%%%%%%%%%%%%%%%%%%%%%%%%%%%%%%

\section{Channel simulation}\label{s_sqc}

In this section, we extend the results of \cite[Section 2]{ght}, showing that 
the QNS correlations of a given type can be composed in a way consistent 
with the quantum simulation paradigm introduced in \cite{dw}. 
We will need a slight extension of a lemma from \cite{tl}. 

\begin{lemma}\label{twisted_comp}
Let $X, Y, Z$ be finite sets, $H$ and $K$ be Hilbert spaces, 
and $E \in M_{X}\otimes M_{Y}\otimes \cl{B}(H)$ and $F \in M_{Y}\otimes M_{Z}\otimes \cl{B}(K)$ be stochastic operator matrices. Set
$$G_{x, x', z, z'} = \sum\limits_{y, y' \in Y}F_{y, y', z, z'}\otimes E_{x, x', y, y'}, \;\;\;\; x, x' \in X, \; z, z' \in Z,$$
and
$$G_{x, x', z, z'}^{\tau} = \sum\limits_{y, y' \in Y}E_{x, x', y, y'}\otimes F_{y, y', z, z'}, \;\;\;\; x, x' \in X, \; z, z' \in Z.$$
\noindent Then 
$G := (G_{x, x', z, z'})_{x, x', z, z'}$ 
and 
$G^{\tau} := (G_{x, x', z, z'}^{\tau})_{x, x', z, z'}$ are stochastic operator matrices in $M_{X}\otimes M_{Z} \otimes \cl{B}(H\otimes K)$. 
\end{lemma}

\begin{proof}
The statement concerning $G$ is
contained in \cite[Lemma 10.16]{tl}. 

By \cite[Theorem 3.1]{tl}, 
there exist a Hilbert space $\tilde{H}$ (resp. $\tilde{K}$) and 
a block operator isometry $V = (V_{y, x})_{y, x}$ (resp. $W = (W_{z, y})_{z, y}$) 
from $H^{X}$ (resp. $K^{Y}$) to $\tilde{H}^{Y}$ (resp. $\tilde{K}^{Z}$), such that
$$E_{x, x', y, y'} = V_{y, x}^{*}V_{y', x'}, \;\;\;\; F_{y, y', z, z'} = W_{z, y}^{*}W_{z', y'}$$
for all $x, x' \in X, y, y' \in Y$ and $z, z' \in Z$.
Set
$$U_{z,x}  = \sum_{y\in Y} V_{y, x}\otimes W_{z, y}, \;\;\;\; x \in X, z \in Z;$$
then 
\begin{eqnarray*}
\sum_{z\in Z} U_{z,x}^*U_{z,x'} 
& = & 
\sum_{z\in Z} \left(\sum_{y\in Y} V_{y, x}\otimes W_{z, y}\right)^*
\left(\sum_{y'\in Y} V_{y', x'}\otimes W_{z, y'}\right)\\
& = & 
\sum_{z\in Z} \sum_{y,y'\in Y} V_{y, x}^*V_{y', x'} \otimes W_{z, y}^*W_{z, y'}\\
& = & 
\sum_{y,y'\in Y} V_{y, x}^*V_{y', x'} \otimes \left(\sum_{z\in Z} W_{z, y}^*W_{z, y'}\right)\\
& = & 
\delta_{y,y'}\sum_{y,y'\in Y} V_{y, x}^*V_{y', x'} \otimes I_K 
= 
\delta_{x,x'} I_H\otimes I_K.
\end{eqnarray*}
Thus, $(U_{z, x})_{z, x}$ is an isometry from $(H\otimes K)^{X}$ to $(\tilde{H}\otimes \tilde{K})^{Z}$ 
with $U_{z, x}^{*}U_{z', x'} = G_{x, x', z, z'}^{\tau}$, $x,x'\in X$, $z,z'\in Z$. 
\end{proof}

The authors of \cite{tl} call the stochastic operator matrix $G$ defined in Lemma \ref{twisted_comp}
the \textit{composition} of $E$ and $F$, denoted $F\circ E$; analogously, we call the stochastic operator matrix 
$G^{\tau}$ the \textit{twisted composition} of $E$ and $F$, and denote it by $F\circ^{\tau} E$. 

%We next 
%recall the definition of quantum simulation of a channel, simulated from a given channel with the 
%assistance of a QNS correlation \cite{dw}. 
Let $\Gamma$ be a QNS correlation 
over the quadruple $(X_{2}, Y_{1}, X_{1}, Y_{2})$ and $\cl{E}: M_{X_{1}}\rightarrow M_{Y_{1}}$
be a quantum channel. Writing $\Gamma = \sum_{i=1}^m \Phi_i\otimes \Psi_i$, where 
$\Phi_i : M_{X_2}\to M_{X_1}$ and $\Psi_i : M_{Y_1}\to M_{Y_2}$ are some linear maps 
(that are not necessarily completely positive), 
we let $\Gamma[\cl{E}] : M_{X_{2}}\rightarrow M_{Y_{2}}$ be linear map, defined by letting
\begin{equation}\label{eq_defsim}
\Gamma[\cl{E}] = \sum_{i=1}^m \Psi_i \circ \cl E \circ \Phi_i.
\end{equation}
It was shown in \cite{dw} that $\Gamma[\cl{E}]$ is a quantum channel, called therein 
the channel \emph{simulated from $\cl E$ with the assistance of $\Gamma$}. 
We call $\Gamma$ a \emph{simulator}, and 
write $(X_{1}\rightarrow Y_{1}) \stackrel{\Gamma}{\rightarrow} (X_{2}\rightarrow Y_{2})$.

Suppose that $(X_{1}\rightarrow Y_{1})\stackrel{\Gamma_{1}}{\rightarrow} (X_{2}\rightarrow Y_{2})$ 
and $(X_{2}\rightarrow Y_{2})\stackrel{\Gamma_{2}}{\rightarrow}(X_{3}\rightarrow Y_{3})$. 
We let $\Gamma_{2}\ast \Gamma_{1}: M_{X_{3}Y_{1}}\rightarrow M_{X_{1}Y_{3}}$ be the linear map
such that 
\begin{eqnarray*}
& &
\hspace{-1.4cm} 
\left\langle(\Gamma_{2}\ast\Gamma_{1})(\epsilon_{x_{3}, x_{3}'}\otimes \epsilon_{y_{1}, y_{1}'}),
\epsilon_{x_{1}, x_{1}'}\otimes \epsilon_{y_{3}, y_{3}'}\right\rangle = \\ 
& & 
\hspace{-0.6cm} 
\sum_{y_{2}, y_{2}'\in Y_2} \sum_{x_{2}, x_{2}'\in X_2}
\Gamma_{1}(x_{1}x_{1}', y_{2}y_{2}'|x_{2}x_{2}', y_{1}y_{1}')
\Gamma_{2}(x_{2}x_{2}', y_{3}y_{3}'|x_{3}x_{3}', y_{2}y_{2}')
\end{eqnarray*}

\begin{lemma}\label{cp_tensor_map}
For a finite set $X$, the linear functional $\phi: M_{X}\otimes M_{X}\rightarrow \bb{C}$, given by 
\begin{gather*}
	P\otimes Q \mapsto {\rm Tr}(PQ^{\rm t}), \ \ \ P,Q\in M_X,
\end{gather*}
is positive. 
\end{lemma}
\begin{proof}
Let $\xi = \sum_{x,x'\in X}\lambda_{x,x'} e_{x}\otimes e_{x'} \in \bb{C}^{X}\otimes \bb{C}^{X}$;
then 
\begin{gather*}
	\xi\xi^{*} = \sum\limits_{x,x'\in X} \sum\limits_{y,y'\in X}
	\lambda_{x,x'}\overline{\lambda_{y,y'}}\epsilon_{x,y}\otimes \epsilon_{x',y'}, 
\end{gather*}
and hence
\begin{eqnarray*}
\phi(\xi\xi^{*})
& = &
\sum\limits_{x,x'\in X}\sum\limits_{y,y'\in X}\lambda_{x,x'}\overline{\lambda_{y,y'}}
{\rm Tr}(\epsilon_{x,y}\epsilon_{y'x'}) 
= 
\sum\limits_{x\in X}\sum\limits_{y\in X}\lambda_{x,x}\overline{\lambda_{y,y}}\\
& = & 
\bigg|\sum\limits_{x\in X}\lambda_{x,x}\bigg|^{2} \geq 0.
\end{eqnarray*}
The claim follows from the fact that any $T \in (M_{X}\otimes M_{X})^{+}$ 
%\marginpar{\tiny IT. A sum of positive rank ones, not a just any linear combination!}
is a sum of positive rank-one operators.
\end{proof}

%\smallskip
%\noindent\textbf{Remark.} 
%\marginpar{\tiny IT. The compositions in the formula do not match! 
%Perhaps the formula is not needed -- remove the Remark if so.}
%For a QNS correlation $\Gamma$ over $(X_{2}, Y_{1}, X_{1}, Y_{2})$ and a channel $\cl{E}$ over $(X_{1}, Y_{1})$, we recall the alternative characterization of quantum simulated channel $\Gamma[\cl{E}]$ via Equation (12) in \cite[Section 2.4]{dw}:
%\begin{equation}\label{quantum_sim}
%	\Gamma[\cl{E}](X) = {\rm Tr}_{Y_{1}Y_{1}}\bigg\{\bigg[\bigg(({\rm id}_{Y_{2}Y_{1}}\otimes \cl{E})\circ ({\rm id}_{Y_{1}}\otimes \Gamma)\bigg)(X\otimes J_{Y_{1}})\bigg]\bigg[I_{Y_{2}}\otimes J_{Y_{1}}\bigg]\bigg\}.
%\end{equation}

\begin{theorem}\label{comp_thm}
If $(X_{1}\rightarrow Y_{1})\stackrel{\Gamma_{1}}{\rightarrow} (X_{2}\rightarrow Y_{2})$ and $(X_{2}\rightarrow Y_{2})\stackrel{\Gamma_{2}}{\rightarrow}(X_{3}\rightarrow Y_{3})$ then
$(X_{1}\rightarrow Y_{1})\stackrel{\Gamma_{2}\ast\Gamma_{1}}{\longrightarrow}(X_{3}\rightarrow Y_{3})$.
Furthermore, 
\begin{enumerate}
	\item[(i)] if ${\rm t} \in \{\rm loc, q, qa, qc, ns\}$ and $\Gamma_{i} \in \cl{Q}_{\rm t}$ for $i = 1, 2$, then $\Gamma_{2}\ast\Gamma_{1} \in \cl{Q}_{\rm t}$;
	\item[(ii)] if $\cl{E}: M_{X_{1}}\rightarrow M_{Y_{1}}$ is a quantum channel then $(\Gamma_{2}\ast \Gamma_{1})[\cl{E}] = \Gamma_{2}[\Gamma_{1}[\cl{E}]]$. 
\end{enumerate}
\end{theorem}

\begin{proof}
(i) 
Let $\Gamma_{1}, \Gamma_{2} \in \cl{Q}_{\rm ns}$. 
If $C_{1} := (\Gamma_{1}(x_{1}x_{1}', y_{2}y_{2}'|x_{2}x_{2}', y_{1}y_{1}'))$ is the Choi matrix of $\Gamma_{1}$ then 
$C_{1}$ satisfies the following conditions:
\begin{enumerate}
	\item[(a)] $C_{1} \in M_{X_{2}Y_{1}X_{1}Y_{2}}^{+}$;
	\item[(b)] there exists $c_{y_{2}, y_{2}'}^{y_{1}, y_{1}'} \in \bb{C}$ such that 
	$\sum_{x_{1}\in X_1}\Gamma_{1}(x_{1}x_{1}, y_{2}y_{2}'|x_{2}x_{2}', y_{1}y_{1}') = \delta_{x_{2}, x_{2}'}c_{y_{2}y_{2}'}^{y_{1}, y_{1}'}, \; y_{1}, y_{1}' \in Y_{1}, y_{2}, y_{2}' \in Y_{2}$;
	\item[(c)] there exists $d_{x_{1}x_{1}'}^{x_{2}, x_{2}'} \in \bb{C}$ such that 
	$\sum_{y_{2}\in Y_2}\Gamma_{1}(x_{1}x_{1}', y_{2}y_{2}|x_{2}x_{2}', y_{1}y_{1}') 
	= \delta_{y_{1}, y_{1}'}d_{x_{1}, x_{1}'}^{x_{2}, x_{2}'}, \; x_{1}, x_{1}' \in X_{1}, x_{2}, x_{2}' \in X_{2}$
\end{enumerate}
(see e.g. \cite[Theorem 6.2]{tl}), and that similar conditions are satisfied for 
the Choi matrix $C_{2} := (\Gamma_{2}(x_{2}x_{2}', y_{3}y_{3}'|x_{3}x_{3}', y_{2}y_{2}'))$ of $\Gamma_{2}$,
with scalars $c_{y_{3}, y_{3}'}^{y_{2}, y_{2}'}$ and $d_{x_{2}x_{2}'}^{x_{3}, x_{3}'}$. 
For $y_{1}, y_{1}' \in Y_{1}, y_{3}, y_{3}' \in Y_{3}, x_{1}, x_{1}' \in X_{1}$ and $x_{3}, x_{3}' \in X_{3}$, set
\begin{gather*}
	c_{y_{3}, y_{3}'}^{y_{1}, y_{1}'} := \sum\limits_{y_{2}, y_{2}' \in Y_2}c_{y_{2}, y_{2}'}^{y_{1}, y_{1}'}c_{y_{3}, y_{3}'}^{y_{2}, y_{2}'} \ \mbox{ and } \ 
	d_{x_{1}, x_{1}'}^{x_{3}, x_{3}'} := \sum\limits_{x_{2}, x_{2}' \in X_2}d_{x_{2}, x_{2}'}^{x_{3}, x_{3}'}d_{x_{1}, x_{1}'}^{x_{2}, x_{2}'}.
\end{gather*}
\noindent Let $C$ be the Choi matrix for $\Gamma_{2}\ast \Gamma_{1}$; thus, 
$$C = 
\left(\sum\limits_{x_{2}, x_{2}' \in X_2}\sum\limits_{y_{2}, y_{2}' \in Y_2}
\Gamma_{2}(x_{2}x_{2}', y_{3}y_{3}'|x_{3}x_{3}', y_{2}y_{2}')
\Gamma_{1}(x_{1}x_{1}', y_{2}y_{2}'|x_{2}x_{2}', y_{1}y_{1}')\right)\hspace{-0.1cm}.$$  
We claim that $C \in M_{X_{3}Y_{1}X_{1}Y_{3}}^{+}$; indeed, if $\tilde{C} := C_{1}\otimes C_{2}$ then 
$\tilde{C}$ is a positive matrix in $M_{X_{2}Y_{1}X_{1}Y_{2}X_{3}Y_{2}X_{2}Y_{3}}$ and, 
after an appropriate shuffling of terms,
\begin{eqnarray*}
\tilde{C} 
& = &
\sum\limits_{\substack{x_{1}, x_{1}' \\ y_{2}, y_{2}'}}\sum\limits_{\substack{x_{2}, x_{2}' \\ y_{1}, y_{1}'}}\sum\limits_{\substack{\tilde{x}_{2}, \tilde{x}_{2}' \\ y_{3}, y_{3}'}}\sum\limits_{\substack{x_{3}, x_{3}' \\ \tilde{y}_{2}, \tilde{y}_{2}'}}\Gamma_{2}(\tilde{x}_{2}\tilde{x}_{2}', y_{3}y_{3}'|x_{3}x_{3}', \tilde{y}_{2}\tilde{y}_{2}')\Gamma_{1}(x_{1}x_{1}', y_{2}y_{2}'|x_{2}x_{2}', y_{1}y_{1}') \\
& &
\;\;\;\;\;\;\;\; \times (\epsilon_{x_{2}x_{2}'}\otimes \epsilon_{\tilde{x}_{2}\tilde{x}_{2}'}\otimes \epsilon_{y_{2}y_{2}'}\otimes \epsilon_{\tilde{y}_{2}\tilde{y}_{2}'}\otimes \epsilon_{x_{3}x_{3}'}\otimes \epsilon_{y_{1}y_{1}'}\otimes \epsilon_{x_{1}x_{1}'}\otimes \epsilon_{y_{3}y_{3}'}).
\end{eqnarray*}
By Lemma \ref{cp_tensor_map}, the map
\begin{gather*}
	\phi \otimes \phi\otimes {\rm id}_{X_{3}}\otimes {\rm id}_{Y_{1}}\otimes {\rm id}_{X_{1}}\otimes {\rm id}_{Y_{3}}: M_{X_{2}X_{2}Y_{2}Y_{2}X_{3}Y_{1}X_{1}Y_{3}}\rightarrow M_{X_{3}Y_{1}X_{1}Y_{3}}
\end{gather*}
is completely positive; in addition,
$$(\phi\otimes \phi\otimes {\rm id}_{M_{X_{3}}}\otimes {\rm id}_{M_{Y_{1}}}\otimes {\rm id}_{M_{X_{1}}}\otimes 
{\rm id}_{M_{Y_{3}}})(\tilde{C}) = C.$$
It follows that
$C \in M_{X_{3}Y_{1}X_{1}Y_{3}}^{+}$ as claimed. 

If $y_{1}, y_{1}' \in Y_{1}$ and $y_{3}, y_{3}' \in Y_{3}$ then 
\begin{eqnarray*}
& &
\hspace{-1cm}
\sum\limits_{x_{1}\in X_1} \hspace{-0.2cm} \left( \hspace{-0.01cm}\sum\limits_{x_{2}, x_{2}'\in X_2}
\sum\limits_{y_{2}, y_{2}' \in Y_2}\Gamma_{2}(x_{2}x_{2}', y_{3}y_{3}'|x_{3}x_{3}', y_{2}y_{2}')\Gamma_{1}(x_{1}x_{1}, y_{2}y_{2}'|x_{2}x_{2}', y_{1}y_{1}')\right) \\
& = &
\sum\limits_{y_{2}, y_{2}'\in Y_2}\sum\limits_{x_{2}, x_{2}' \in X_2}
\delta_{x_{2}, x_{2}'}\Gamma_{2}(x_{2}x_{2}', y_{3}y_{3}'|x_{3}x_{3}', y_{2}y_{2}')c_{y_{2}, y_{2}'}^{y_{1}, y_{1}'}\\
& = &
\sum\limits_{y_{2}, y_{2}' \in Y_2}c_{y_{2}, y_{2}'}^{y_{1}, y_{1}'}
\left(\sum\limits_{x_{2}\in X_2}\Gamma_{2}(x_{2}x_{2}, y_{3}y_{3}'|x_{3}x_{3}', y_{2}y_{2}')\right)\\
& = &
\delta_{x_{3}, x_{3}'}\sum\limits_{y_{2}, y_{2}'\in Y_2}c_{y_{2}, y_{2}'}^{y_{1}, y_{1}'}c_{y_{3}, y_{3}'}^{y_{2}, y_{2}'} 
= 
\delta_{x_{3}, x_{3}'}c_{y_{3}, y_{3}'}^{y_{1}, y_{1}'}.
\end{eqnarray*}
\noindent Similarly, for $x_{1}, x_{1}' \in X_{1}, x_{3}, x_{3}' \in X_{3}$ we have
\begin{eqnarray*}
& &
\hspace{-1cm}
\sum\limits_{y_{3}\in Y_3} \hspace{-0.2cm} 
\left( \hspace{-0.01cm}\sum\limits_{x_{2}, x_{2}' \in X_2}\sum\limits_{y_{2}, y_{2}' \in Y_2}
\Gamma_{2}(x_{2}x_{2}', y_{3}y_{3}|x_{3}x_{3}', y_{2}y_{2}')
\Gamma_{1}(x_{1}x_{1}', y_{2}y_{2}'|x_{2}x_{2}', y_{1}y_{1}')\right)\\
& = &
\sum\limits_{x_{2}, x_{2}' \in X_2}\sum\limits_{y_{2}, y_{2}'\in Y_2}\delta_{y_{2}, y_{2}'}
\Gamma_{1}(x_{1}x_{1}', y_{2}y_{2}'|x_{2}x_{2}', y_{1}y_{1}')d_{x_{2}, x_{2}'}^{x_{3}, x_{3}'}\\
& = &
\sum\limits_{x_{2}, x_{2}' \in X_2}d_{x_{2}, x_{2}'}^{x_{3}, x_{3}'}
\left(\sum\limits_{y_{2}\in Y_2}\Gamma_{1}(x_{1}x_{1}', y_{2}y_{2}|x_{2}x_{2}', y_{1}y_{1}')\right)\\
& = &
\delta_{y_{1}, y_{1}'}\sum\limits_{x_{2}, x_{2}' \in X_2}
d_{x_{2}, x_{2}'}^{x_{3}, x_{3}'}d_{x_{1}, x_{1}'}^{x_{2}, x_{2}'}
=
\delta_{y_{1}, y_{1}'}d_{x_{1}, x_{1}'}^{x_{3}, x_{3}'}.
\end{eqnarray*}
\noindent The former equality implies $L_{\omega_{Y_{1}Y_{3}}}(C) \in \cl{L}_{X_{3}X_{1}}$ for every $\omega_{Y_{1}Y_{3}} \in M_{Y_{1}Y_{3}}$, while the latter equality implies $L_{\omega_{X_{3}X_{1}}}(C) \in \cl{L}_{Y_{1}Y_{3}}$ for every $\omega_{X_{3}X_{1}} \in M_{X_{3}X_{1}}$. Thus, 
\begin{gather*}
	C \in (\cl{L}_{X_{3}X_{1}}\otimes \cl{L}_{Y_{1}Y_{3}}) \cap M_{X_{3}Y_{1}X_{1}Y_{3}}^{+};
\end{gather*}
by the injectivity of the minimal operator system tensor product, 
$C \in (\cl{L}_{X_{3}X_{1}}\otimes_{\min} \cl{L}_{Y_{1}Y_{3}})^{+}$. 
It follows as in \cite[Theorem 6.2]{tl} that $\Gamma_{2}\ast \Gamma_{1} \in \cl{Q}_{\rm ns}$.

Suppose that $\Gamma_{i} \in \cl{Q}_{\rm qc}, i = 1, 2$. Let $(E^{(i)}, F^{(i)})$ 
be a commuting pair of stochastic operator matrices acting on a Hilbert space $\cl{H}_{i}$, 
and $\xi_{i}\in \cl{H}_{i}$ be a unit vector, such that 
$\Gamma_{i}$ is represented through $(E^{(i)}, F^{(i)}, \xi_{i})$ as in (\ref{eq_EFp}), $i = 1,2$, 
that is, if 
$E^{(1)} = (E_{x_{2}, x_{2}', x_{1}, x_{1}'}^{(1)})$, $F^{(1)} = (F_{y_{1}, y_{1}', y_{2}, y_{2}'}^{(1)})$, $E^{(2)} = (E_{x_{3}, x_{3}', x_{2}, x_{2}'}^{(2)})$ and $F^{(2)} = (F_{y_{2}, y_{2}', y_{3}, y_{3}'}^{(2)})$, then 
$$	\Gamma_{1}(x_{1}x_{1}', y_{2}y_{2}'|x_{2}x_{2}', y_{1}y_{1}') = \langle E_{x_{2}, x_{2}', x_{1}, x_{1}'}^{(1)}F_{y_{1}, y_{1}', y_{2}, y_{2}'}^{(1)}\xi_{1}, \xi_{1}\rangle$$
and
$$\Gamma_{2}(x_{2}x_{2}', y_{3}y_{3}'|x_{3}x_{3}', y_{2}y_{2}') = \langle E_{x_{3}, x_{3}', x_{2}, x_{2}'}^{(2)}F_{y_{2}, y_{2}', y_{3}, y_{3}'}^{(2)}\xi_{2}, \xi_{2}\rangle.$$
Set $H = H_{1} \otimes H_{2}$, $\xi = \xi_{1}\otimes \xi_{2}$, $E = E^{(1)}\circ E^{(2)}$, and 
$F = F^{(2)}\circ^{\tau} F^{(1)}$. 
By Lemma \ref{twisted_comp}, both $E$ and $F$ are stochastic operator matrices. 
It is easy to check that $(E, F)$ form a commuting pair. 
Writing $E = (E_{x_{3}, x_{3}', x_{1}, x_{1}'})$ and $F = (F_{y_{1}, y_{1}', y_{3}, y_{3}'})$, we note that
\begin{eqnarray*}
& &
\hspace{-0.35cm}
\left\langle (\Gamma_{2}\ast \Gamma_{1})(\epsilon_{x_{3}, x_{3}'}\otimes \epsilon_{y_{1}, y_{1}'}), 
\epsilon_{x_{1}, x_{1}'}\otimes \epsilon_{y_{3}, y_{3}'}\right\rangle\\
& = & 
\hspace{-0.35cm}\sum\limits_{y_{2}, y_{2}'}\sum\limits_{x_{2}, x_{2}'}\langle E_{x_{2}, x_{2}', x_{1}, x_{1}'}^{(1)}F_{y_{1}, y_{1}', y_{2}, y_{2}'}^{(1)}\xi_{1}, \xi_{1}\rangle\langle E_{x_{3}, x_{3}', x_{2}, x_{2}'}^{(2)}F_{y_{2}, y_{2}', y_{3}, y_{3}'}^{(2)}\xi_{2}, \xi_{2}\rangle \\
& = &
\hspace{-0.35cm}\sum%\limits_{y_{2}, y_{2}',x_{2}, x_{2}'}%\sum\limits_{}
\hspace{-0.01cm}\langle(E_{x_{2}, x_{2}', x_{1}, x_{1}'}^{(1)}\hspace{-0.1cm}\otimes E_{x_{3}, x_{3}', x_{2}, x_{2}'}^{(2)})
\hspace{-0.05cm}
(F_{y_{1}, y_{1}', y_{2}, y_{2}'}^{(1)}\hspace{-0.1cm}\otimes F_{y_{2}, y_{2}', y_{3}, y_{3}'}^{(2)})
\xi_{1}\otimes \xi_{2}, \xi_{1}\otimes \xi_{2}\rangle\\
& = &
\hspace{-0.35cm}\langle E_{x_{3}, x_{3}', x_{1}, x_{1}'}F_{y_{1}, y_{1}', y_{3}, y_{3}'}\xi, \xi\rangle;
\end{eqnarray*}
thus, $\Gamma_{2} \ast \Gamma_{1} \in \cl{Q}_{\rm qc}$.

If $\Gamma_{i} \in \cl{Q}_{\rm q}, i = 1, 2$, 
replacing operator products with tensor products as needed in the arguments of the previous paragraph 
shows $\Gamma_{2} \ast \Gamma_{1} \in \cl{Q}_{\rm q}$. 
The continuity of the map $(\Gamma_{1}, \Gamma_{2}) \mapsto \Gamma_{2}\ast \Gamma_{1}$ 
ensures that if 
$\Gamma_{i} \in \cl{Q}_{\rm qa}, i = 1, 2$, then $\Gamma_{2}\ast \Gamma_{1} \in \cl{Q}_{\rm qa}$.

Finally, assume that $\Gamma_{i} \in \cl{Q}_{\rm loc}, i = 1, 2$; 
without loss of generality, suppose that 
$\Gamma_{i} = \Phi_{i}\otimes \Psi_{i}$, 
 for some quantum channels
 $\Phi_{1}: M_{X_{2}}\rightarrow M_{X_{1}}$, $\Phi_{2}: M_{X_{3}}\rightarrow M_{X_{2}}$, 
 $\Psi_{1}: M_{Y_{1}}\rightarrow M_{Y_{2}}$ and $\Psi_{2}: M_{Y_{2}}\rightarrow M_{Y_{3}}$;
we then have that
$$	\Gamma_{2}\ast \Gamma_{1} = \left(\Phi_{1}\circ \Phi_{2}\right)
	\otimes \left(\Psi_{2}\circ\Psi_{1}\right).$$
%Indeed: if $\Gamma \in \cl{Q}_{\rm loc}$ is of the form $\Gamma = \Phi\otimes \Psi$, letting $E, F$ be the Choi matrices for $\Phi, \Psi$ by \cite[Remark 4.10]{tl} we know $\Gamma = \Gamma_{E, 1}\otimes \Gamma_{F, 1} = \Gamma_{E \odot F, 1}$. Applying this type of argument to each $\Phi_{j}^{(i)}\otimes \Psi_{j}^{(i)}$ (for $j \in \{k, \ell\}$ and $i = 1, 2$) we can make a similar argument as in the $\cl{Q}_{\rm q}$ case to conclude the above. Therefore, $\Gamma_{2} \ast \Gamma_{1} \in \cl{Q}_{\rm loc}$. 

(ii) 
For $x_{3}, x_{3}' \in X_{3}$, we have
\begin{eqnarray*}
& &
\Gamma_{2}[\Gamma_{1}[\cl{E}]](\epsilon_{x_{3}, x_{3}'})\\
& = &
\hspace{-0.3cm}\sum
\hspace{-0.05cm}\Gamma_{2}(x_{2}x_{2}', y_{3}y_{3}'|x_{3}x_{3}', y_{2}y_{2}')\Gamma_{1}(x_{1}x_{1}',y_{2}y_{2}'|x_{2}x_{2}', y_{1}y_{1}')
%\;\;\;\;\;\;\;\;\;\;\;\;\;\;\;\;\;\times
\cl{E}(y_{1}y_{1}'|x_{1}x_{1}')\epsilon_{y_{3}, y_{3}'},
\end{eqnarray*}
where the summation is over all 
$y_{1}, y_{1}',y_{2}, y_{2}',y_{3}, y_{3}',x_{1}, x_{1}',x_{2}, x_{2}'$;
by the proof of (i), 
$\Gamma_{2}[\Gamma_{1}[\cl{E}]] = (\Gamma_{2}\ast \Gamma_{1})[\cl{E}]$. 
\end{proof}

%%%%%%%%%%%%%%%%%%%%%%%%%%%%%%%%%%%%%%%%%%%%%%%%%
%%%%%%%%%%%%%%%%%%%%%%%%%%%%%%%%%%%%%%%%%%%%%%%%%
%%%%%%%%%%%%%%%%%%%%%%%%%%%%%%%%%%%%%%%%%%%%%%%%%

\section{The homomorphism types}\label{s_quantize}

For $\cl{U} \subseteq \overline{H}\otimes K$, we set $\tilde{\cl{U}} = \theta(\cl{U})$ and, conversely, for a subspace $\tilde{\cl{U}} \subseteq \cl{L}(H, K)$ we write $\cl{U} = \theta^{-1}(\tilde{\cl{U}})$. 
%and $\supp(A)$ be the closure of the range of $A$. 
%Let $\# : \bb{C}^{X}\rightarrow \bb{C}^{X}$ be the conjugation map on $\bb{C}^{X}$ with respect to the standard orthonormal basis; thus, $\#\left(\sum_{x\in X} \lambda_x e_x\right) = \sum_{x\in X} \overline{\lambda}_x e_x$
%whenever $\lambda_x\in \bb{C}$, $x\in X$. 

\begin{definition}\label{q_hypergraph}
Let $X, Y$ be finite sets. 
A subspace $\cl{U} \subseteq \overline{\bb{C}}^{X}\otimes \bb{C}^{Y}$ will be called
a \emph{quantum hypergraph over $(X,Y)$}.
\end{definition}

\begin{remark}\label{r_clavquant}
\rm 
Let $X$ be a finite set. Recall that a (classical) \emph{hypergraph} with vertex set $X$ is a set of subsets of $X$,
called \emph{hyperedges}. 
Given a(nother) set $Y$ we can view a subset $E\subseteq X\times Y$ as a hypergraph with vertex set $X$ 
and hyperedges $E_y = \{x\in X : (x,y)\in E\}$, $y\in Y$; by abuse of notation, we 
call the subset $E$ a hypergraph. We further call a hypergraph $E\subseteq X\times Y$ \emph{full}
if for every $x\in X$ there exists $y\in Y$ such that $(x,y)\in E$. 
Given a hypergraph $E\subseteq X\times Y$, we let 
$$\cl U_E = {\rm span}\{\overline{e_x} \otimes e_y : (x,y)\in E\}$$
and view $\cl U_E$ as a quantum hypergraph over $(X,Y)$. We call quantum hypergraphs of this form \textit{classical}. 
\end{remark}

It is easy to see that a hypergraph $\cl{U} \subseteq \overline{\bb{C}}^{X}\otimes \bb{C}^{Y}$ is classical 
if and only if the subspace $\tilde{\cl U} \subseteq \cl{L}(\bb{C}^{X}, \bb{C}^{Y})$ 
is a $\cl{D}_{Y},\cl{D}_{X}$-bimodule. 
Indeed, 
given a $\cl{D}_{Y},\cl{D}_{X}$ bimodule $\tilde{\cl{U}}$, let 
\begin{gather*}
	E_{\tilde{\cl{U}}} = \{(x, y): \; \exists \; A\in \tilde{\cl U}\; \text{such that} \; \langle Ae_y,e_x\rangle \neq 0\};
\end{gather*}
it is then easily verified that $\tilde{\cl{U}} = \tilde{\cl{U}}_{E_{\tilde{\cl{U}}}}$.

Let $X, Y$ be finite sets. A linear map 
$\Phi : M_{X}\rightarrow M_{Y}$ is called a \emph{quantum channel} if 
$\Phi$ is completely positive and trace-preserving.
Every quantum channel $\Phi$ has a \emph{Kraus representation} 
\cite[Proposition 2.20]{watrous}, that is,
\begin{equation}\label{eq_Kraus}
	\Phi(T) = \sum\limits_{k=1}^m A_{k}TA_{k}^{*}, \ \ \ T\in M_X,
\end{equation}
for some $m\in \bb{N}$ and operators $A_i\in \cl B(\bb{C}^X,\bb{C}^Y)$, $i \in [m]$, 
such that $\sum_{i=1}^m A_i^* A_i$ $=$ $I_X$. 
The \emph{Kraus space} $\tilde{\cl{K}}_{\Phi}$ of $\Phi$ is defined by letting
\begin{gather*}
	\tilde{\cl{K}}_{\Phi} := {\rm span}\{A_{k}: \; k\in [m]\};
\end{gather*}
we note that, by the uniqueness clause of Stinespring's Theorem \cite[Corollary 2.23]{watrous}, 
$\tilde{\cl{K}}_{\Phi}$ is 
independent of the particular Kraus representation (\ref{eq_Kraus}) of $\Phi$.

\begin{remark}\label{classic_supp}
\rm 
Let $\cl N : \cl{D}_{X}\rightarrow \cl{D}_{Y}$ be a classical channel. Then 
\begin{gather*}
	\Gamma_{\cl N} (T) = \sum_{x\in X} \sum_{y\in Y} \cl N(y|x)\epsilon_{y,x}T\epsilon_{x,y}, 
	\ \ \ T\in M_X, 
\end{gather*}
and it follows that 
%\begin{eqnarray*}
$$\tilde{\cl{K}}_{\Gamma_{\cl N}} \hspace{-0.05cm} = \hspace{-0.05cm} 
{\rm span}\{\hspace{-0.05cm}\sqrt{\cl N(y|x)}\epsilon_{y, x} : 
x \hspace{-0.02cm}\in\hspace{-0.02cm} X, y \hspace{-0.02cm}\in \hspace{-0.02cm}Y\} 
%= {\rm span}\{\epsilon_{y,x}: \; \cl N(y|x) \neq 0\}
= {\rm span}\{\epsilon_{y,x}: (x, y)\hspace{-0.02cm} \in\hspace{-0.02cm} {\rm supp}(\cl N)\}\hspace{-0.02cm}.$$
%\end{eqnarray*}
\end{remark}

Here, and in the sequel, we fix finite sets $X_{i}, Y_{i}$, $i = 1, 2$. Let $\cl{U}_{1} \subseteq \bb{C}^{X_{1}}\otimes \overline{\bb{C}}^{Y_{1}}$ and $\cl{U}_{2} \subseteq \overline{\bb{C}}^{X_{2}}\otimes \bb{C}^{Y_{2}}$ be quantum hypergraphs. Set 
%\begin{equation}\label{quasi_hyp}
$$	\cl{U}_{1}\hspace{-0.1cm}\Rightarrow\hspace{-0.05cm} \cl{U}_{2} := 
(\cl{U}_{1}\otimes \cl{U}_{2})
+ \left(\cl{U}_{1}^{\perp}\otimes (\overline{\bb{C}}^{X_{2}}\otimes \bb{C}^{Y_{2}})\right)$$
%\end{equation}
and
$$	\cl{U}_{1}\hspace{-0.1cm}\Leftrightarrow \hspace{-0.05cm}\cl{U}_{2} 
	:= (\cl{U}_{1}\otimes \cl{U}_{2})+(\cl{U}_{1}^{\perp}\otimes \cl{U}_{2}^{\perp}),$$
considered as quantum hypergraphs in 
$(\bb{C}^{X_{1}}\otimes \overline{\bb{C}}^{Y_{1}})\otimes (\overline{\bb{C}}^{X_{2}}\otimes \bb{C}^{Y_{2}})$.

Let 
$$\sigma: \bb{C}^{X_{1}}\otimes\overline{\bb{C}}^{Y_{1}}\otimes \overline{\bb{C}}^{X_{2}}\otimes \bb{C}^{Y_{1}}\rightarrow \overline{\bb{C}}^{X_{2}}\otimes \overline{\bb{C}}^{Y_{1}}\otimes \bb{C}^{X_{1}}\otimes \bb{C}^{Y_{2}}$$ 
be the flip operator defined on the elementary tensors by
%\begin{gather*}
$	\sigma(\xi_1\otimes \overline{\eta}_1\otimes\overline{\xi}_2\otimes \eta_2) 
	= \overline{\xi}_2\otimes \overline{\eta}_1\otimes\xi_1\otimes \eta_2.$
For quantum hypergraphs $\cl{U}_{1} \subseteq \bb{C}^{X_{1}}\otimes \overline{\bb{C}}^{Y_{1}}$ and $\cl{U}_{2} \subseteq \overline{\bb{C}}^{X_{2}}\otimes \bb{C}^{Y_{2}}$, we set 
$$\cl{U}_{1}\hspace{-0.1cm}\rightarrow\hspace{-0.05cm} \cl{U}_{2} := 
\sigma(\cl{U}_{1}\hspace{-0.1cm}\Rightarrow \hspace{-0.05cm}\cl{U}_{2}) 
\ \mbox{ and } \ 
\cl{U}_{1}\hspace{-0.1cm}\leftrightarrow\hspace{-0.05cm} \cl{U}_{2} := 
\sigma(\cl{U}_{1}\hspace{-0.1cm}\Leftrightarrow \hspace{-0.05cm}\cl{U}_{2}).$$

%\end{gather*}
%Again taking the setup of \cite{ght} as an example, we want to highlight the collection of all quantum channels fitting a particular hypergraph.
Given a quantum hypergraph $\cl{K} \subseteq \overline{\bb{C}}^{X}\otimes \bb{C}^{Y}$, we let
$$	\cl{Q}(\cl{K}) := \left\{\Gamma: M_{X}\rightarrow M_{Y} \; : \; \text{a quantum channel with } 
\tilde{\cl{K}}_{\Gamma} \subseteq \tilde{\cl{K}}\right\};$$
if $\Gamma\in \cl{Q}(\cl{K})$, we say that $\Gamma$ \textit{fits} $\cl{K}$.

\begin{definition}\label{d_quasihom}
Let ${\rm t} \in \{{\rm loc, q, qa, qc, ns}\}$, and 
$\cl{U}_{1} \subseteq \bb{C}^{X_{1}}\otimes \overline{\bb{C}}^{Y_{1}}$
and 
$\cl{U}_{2} \subseteq \overline{\bb{C}}^{X_{2}}\otimes \bb{C}^{Y_{2}}$ be quantum hypergraphs. 
We say that $\cl{U}_{1}$ is ${\rm t}$-quasi-homomorphic (resp. ${\rm t}$-homomorphic) to $\cl{U}_{2}$,
and write $\cl{U}_{1} \leadsto_{\rm t} \cl{U}_{2}$ (resp. $\cl{U}_{1} \to_{\rm t} \cl{U}_{2}$), 
if there exists a quantum channel 
$\Gamma: M_{X_{2}Y_{1}}\rightarrow M_{X_{1}Y_{2}}$ with $\Gamma \in \cl{Q}_{\rm t}$, 
such that $\Gamma \in \cl{Q}(\cl{U}_{1}\hspace{-0.1cm}\rightarrow \hspace{-0.05cm}\cl{U}_{2})$
(resp. $\Gamma \in \cl{Q}(\cl{U}_{1}\hspace{-0.1cm}\leftrightarrow \hspace{-0.05cm}\cl{U}_{2})$). 
\end{definition}

It is clear that the relation 
$\cl{U}_{1} \rightarrow_{\rm t} \cl{U}_{2}$ implies the relation $\cl{U}_{1} \leadsto_{\rm t} \cl{U}_{2}$
for any choice of ${\rm t} \in \{\rm loc, q, qa, qc, ns\}$. 
In Proposition \ref{p_twisted-Choi} below, we characterise the constraint condition on the channel $\Gamma$
appearing in Definition \ref{d_quasihom}, in terms of a suitable modification of the Choi matrix of $\Gamma$.
More specifically, let  
$\Gamma : M_{X}\rightarrow M_{Y}$ be a linear map. 
Set
$\tilde{C}_{\Gamma} := \sum_{x, x' \in X}\overline{\epsilon}_{x',x}\otimes \Gamma(\epsilon_{x,x'})$.

\begin{lemma}\label{l_tChoi}
Let $X$ and $Y$ be finite sets, $A\in \cl L(\bb{C}^X,\bb{C}^Y)$ and $\Gamma : M_X\to M_Y$ be the map, given by 
$\Gamma(T) = ATA^*$. 
If $\zeta_A = \theta^{-1}(A)$ then $\tilde{C}_{\Gamma} = \zeta_A\zeta_A^*$. 
\end{lemma}

\begin{proof}
We note that $\overline{e}_x$, $x\in X$, are the elements of the canonical basis of $\overline{\bb{C}}^X$. 
Note also that, if $\xi = \sum_{x\in X} \lambda_x e_x\in \bb{C}^X$ then 
$\overline{\xi} = \sum_{x\in X} \overline{\lambda}_x \overline{e}_x$. 
Further, $\overline{\epsilon}_{x,x'} = \overline{e}_{x'}\overline{e}_x^*$;
indeed, for $x,x',a,a'\in X$, we have
$$\overline{\epsilon}_{x,x'}(\overline{e}_a) (e_{a'})  = 
\overline{e}_a(\epsilon_{x,x'}(e_{a'})) = \delta_{x',a'} \overline{e}_a(e_x) = 
\delta_{x',a'}\delta_{x,a}
=
(\overline{e}_{x'}\overline{e}_x^*)(\overline{e}_a)(e_{a'}),$$
and the claim follows by linearity. 

Let $\xi_i\in \bb{C}^X$, $\eta_i\in \bb{C}^Y$, $A_i = \eta_i\xi_i^*$, $i = 1,2$, and 
$\Lambda : M_X\to M_Y$ be the linear map, given by 
$\Lambda(T) = A_1 T A_2^*$. 
Write $\xi_i = \sum_{x\in X} \lambda_x^{(i)} e_x$, where $\lambda_x^{(i)}\in \bb{C}$,
$x\in X$, $i = 1,2$. 
Set $\zeta_i = \theta^{-1}(A_i)$; thus, $\zeta_i = \overline{\xi}_i\otimes\eta_i$, $i = 1,2$, and hence 
$$\zeta_1\zeta_2^* 
= 
\overline{\xi}_1\overline{\xi}_2^*\otimes\eta_1\eta_2^*
= \hspace{-0.2cm} \sum_{x,x'\in X} \overline{\lambda}_x^{(1)}  \lambda_{x'}^{(2)} \overline{e}_x  \overline{e}_{x'}^{*} \otimes\eta_1\eta_2^*\\
= \hspace{-0.2cm} 
\sum_{x,x'\in X} \overline{\lambda}_x^{(1)}  \lambda_{x'}^{(2)} \overline{\epsilon}_{x',x} \otimes\eta_1\eta_2^*.$$
Thus, 
\begin{eqnarray*}
\tilde{C}_{\Lambda} 
& = & 
\sum_{x, x' \in X}\overline{\epsilon}_{x',x}\otimes A_1\epsilon_{x,x'} A_2^*
=
\sum_{x, x' \in X}\overline{\epsilon}_{x',x}\otimes (A_1e_x) (A_2e_{x'})^*\\
& = & 
\sum_{x, x' \in X}
\overline{\lambda}_x^{(1)} \lambda_{x'}^{(2)} \overline{\epsilon}_{x',x}\otimes \eta_1\eta_2^*
= \zeta_1\zeta_2^*.
\end{eqnarray*}

Write $A = \sum_{i=1}^m A_i$, where $A_i$ has rank one, $i = 1,\dots,m$. 
Then $\Gamma(T) = \sum_{i,j=1}^m A_i T A_j^*$. By additivity,
$$\tilde{C}_{\Gamma} = \sum_{i,j=1}^m \zeta_{A_i}\zeta_{A_j}^* = 
\left(\sum_{i=1}^m \zeta_{A_i}\right) \left(\sum_{j=1}^m\zeta_{A_j}\right)^* 
= \zeta_A\zeta_A^*.$$
\end{proof}

\begin{proposition}\label{p_twisted-Choi}
Let $\cl{U} \subseteq \overline{\bb{C}}^{X_{2}Y_{1}}\hspace{-0.05cm} \otimes \bb{C}^{X_{1}Y_{2}}$ be a quantum hypergraph, and 
%\marginpar{\tiny IT. The support is identical to the range of the operator -- correct?}
% GH. Yes- correct, based on our conversation on 4/25
$\Gamma: M_{X_{2}Y_{1}}\rightarrow M_{X_{1}Y_{2}}$ be a quantum channel. 
Then ${\rm ran}(\tilde{C}_{\Gamma}) \subseteq \cl{U}$ if and only if $\tilde{\cl{K}}_{\Gamma} \subseteq \tilde{\cl{U}}$. 
\end{proposition}

\begin{proof}
Assume that  
$\Gamma(T) = \sum\limits_{k=1}^m A_{k}TA_{k}^{*}$ in a Kraus representation, 
where $A_{k} \in \cl{L}(\bb{C}^{X_{2}Y_{1}}, \bb{C}^{X_{1}Y_{2}})$, $k = 1,\dots,m$. 
Let $\zeta_k = \theta^{-1}(A_{k})$, $k = 1,\dots,m$. 
By Lemma \ref{l_tChoi} and additivity, 
$\tilde{C}_{\Gamma} = \sum\limits_{k=1}^m \zeta_k\zeta_k^*$; in particular, $\tilde{C}_{\Gamma}$ is positive. 
It now follows that 
${\rm ran}\; \tilde{C}_{\Gamma} \subseteq \cl{U}$ if and only if 
$\zeta_k  \in \cl{U}$ for each $k$; on the other hand, the latter condition is 
equivalent to $\tilde{\cl{K}}_{\Gamma} \subseteq \tilde{\cl{U}}$.
\end{proof}

We next show that the notions, introduced in Definition \ref{d_quasihom},
are quantum versions of the different types quasi-homomorphisms and homomorphisms
of classical hypergraphs, introduced in \cite{ght}. 
If $E_i\subseteq X_i\times Y_i$, $i = 1,2$, are (classical) hypergraphs, following \cite{ght}, 
we let 
$$E_1\hspace{-0.1cm} \rightarrow\hspace{-0.05cm} E_2 = \{(x_2,y_1,x_1,y_2) \in X_2Y_1\times X_1Y_2 
: (x_1,y_1)\in E_1\Rightarrow (x_2,y_2)\in E_2\}$$
and 
$$E_1\hspace{-0.1cm} \leftrightarrow\hspace{-0.05cm} E_2 = \{(x_2,y_1,x_1,y_2) \in X_2Y_1\times X_1Y_2 
: (x_1,y_1)\in E_1\Leftrightarrow (x_2,y_2)\in E_2\}.$$
If ${\rm t}\in \{{\rm loc, q, qa, qc, ns}\}$, 
we write \cite{ght} 
$E_1\leadsto_{\rm t} E_2$ (resp. $E_1\rightarrow_{\rm t} E_2$) if there exists
$\cl N\in \cl C_{\rm t}$ such that $\supp(\cl N)\subseteq E_1\hspace{-0.1cm} \rightarrow\hspace{-0.05cm} E_2$
(resp. $\supp(\cl N)\subseteq E_1\hspace{-0.1cm} \leftrightarrow\hspace{-0.05cm} E_2$).

\begin{proposition}\label{classical_to_quantum}
Let ${\rm t} \in \{\rm loc, q, qa, qc, ns\}$ and 
$E_{i} \subseteq X_{i}\times Y_{i}$ be a hypergraph, $i = 1, 2$. Then
\begin{itemize}
	\item[(i)] $E_{1} \leadsto_{\rm t} E_{2}$ if and only if $\overline{\cl{U}}_{E_{1}} \leadsto_{\rm t} 
	\cl{U}_{E_{2}}$, and
	\item[(ii)] $E_{1} \rightarrow_{\rm t} E_{2}$ if and only if 
	$\overline{\cl{U}}_{E_{1}}\rightarrow_{\rm t} \cl{U}_{E_{2}}$. 
\end{itemize}
\end{proposition}

\begin{proof}
(i) First, assume $\overline{\cl{U}}_{E_{1}}\leadsto_{\rm t} \cl{U}_{E_{2}}$ 
via $\Gamma: M_{X_{2}Y_{1}}\rightarrow M_{X_{1}Y_{2}}$. 
Write $\tilde{\Gamma} = \Gamma_{\cl N_{\Gamma}}$; thus, 
$\tilde{\Gamma}: M_{X_{2}Y_{1}}\rightarrow M_{X_{1}Y_{2}}$ is a quantum channel. 
By \cite[Remark 8.1]{tl}, as $\Gamma \in \cl{Q}_{\rm t}$, 
we have that $\cl N_{\Gamma} \in \cl{C}_{\rm t}$ and $\tilde{\Gamma} \in \cl{Q}_{\rm t}$. 
Note that, if 
$\Gamma(T) = \sum\limits_{k=1}^m A_{k}TA_{k}^{*}$, $T\in M_{X_2Y_1}$ 
is a Kraus representation of $\Gamma$ then the set
$$\cl K :=
\{(\epsilon_{x_{1},x_{1}}\otimes \epsilon_{y_{2},y_{2}})A_{k}(\epsilon_{x_{2},x_{2}}\otimes \epsilon_{y_{1},y_{1}}) 
: x_i\in X_i, y_i\in Y_i, i = 1,2, k\in [m]\}$$
of operators in $\cl L(\bb{C}^{X_2Y_1},\bb{C}^{X_1Y_2})$ 
is a family of Kraus operators for both $\cl N_{\Gamma}$ and $\tilde{\Gamma}$.
By Lemma \ref{l_tChoi} and Proposition \ref{p_twisted-Choi},
$\cl K \subseteq \overline{\cl{U}}_{E_{1}} \hspace{-0.1cm}\rightarrow \hspace{-0.05cm}\cl{U}_{E_{2}}$. 
On the other hand, 
\begin{equation}\label{eq_trE1E2}
\overline{\cl{U}}_{E_{1}} \hspace{-0.1cm}\rightarrow \hspace{-0.05cm}\cl{U}_{E_{2}} = 
\cl{U}_{E_{1}\hspace{-0.05cm}\rightarrow \hspace{-0.05cm}E_{2}}.
\end{equation}
It follows that 
$\cl N_{\Gamma}$ fits $E_{1}\hspace{-0.05cm}\rightarrow \hspace{-0.05cm}E_{2}$ and hence 
$E_{1} \leadsto_{\rm t} E_{2}$.

For the converse, assume that $E_{1} \leadsto_{\rm t} E_{2}$ via the 
(classical) NS correlation
$\cl N : \cl{D}_{X_{2}Y_{1}}\rightarrow \cl{D}_{X_{1}Y_{2}}$ of type ${\rm t}$. 
By Remark \ref{classic_supp}, 
\begin{gather*}
	\Gamma_{\cl N}(T) = \sum\limits_{x_{2}, y_{1}}\sum\limits_{x_{1}, y_{2}}\cl N(x_{1}y_{2}|x_{2}y_{1})(\epsilon_{x_{1},x_{2}}\otimes \epsilon_{y_{2},y_{1}})T(\epsilon_{x_{2},x_{1}}\otimes \epsilon_{y_{1},y_{2}})
\end{gather*}
and hence $\Gamma_{\cl N} \in \cl{Q}_{\rm t}$ (see \cite[Remarks 2.2, 4.4, 4.7, 4.11]{tl}). 
By (\ref{eq_trE1E2}), $\Gamma_{\cl N}$ fits 
$\overline{\cl{U}}_{E_{1}} \hspace{-0.1cm}\rightarrow \hspace{-0.05cm}\cl{U}_{E_{2}}$, and hence 
$\overline{\cl{U}}_{E_{1}} \leadsto_{\rm t} \cl{U}_{E_{2}}$.

(ii) is similar, and the details are omitted.
\end{proof}

\begin{lemma}\label{slice_inclusion}
Let 
$\cl{U}_{1} \subseteq \bb{C}^{X_{1}}\otimes \overline{\bb{C}}^{Y_{1}}$
and 
$\cl{U}_{2} \subseteq \overline{\bb{C}}^{X_{2}}\otimes \bb{C}^{Y_{2}}$ be quantum hypergraphs, and 
$\zeta \in \bb{C}^{X_{1}}\otimes \overline{\bb{C}}^{Y_{1}}\otimes \overline{\bb{C}}^{X_{2}}\otimes \bb{C}^{Y_{2}}$.
%\marginpar{\tiny IT. Important change: Elements $\overline{u}_{1}$ and $\overline{u}_{2}$ used 
%instead of $u_1$ and $u_2$; thus, these are considered as linear functionals and not as elements 
%in the underlying Hilbert spaces. 
%The main point here is that $L_{\overline{u}}(v) = \langle v,u\rangle$.
%Check.}
We have that 

\begin{itemize}
\item[(i)] $L_{\overline{u}_{1}}(\zeta) \in \cl{U}_{2}$ for every 
$\overline{u}_{1} \in \overline{\cl{U}}_{1}$ if and only if 
$\zeta \in \cl{U}_{1}\hspace{-0.1cm}\Rightarrow\hspace{-0.05cm}\cl{U}_{2}$; 
\item[(ii)] 
$L_{\overline{u}_{1}}(\zeta) \in \cl{U}_{2}$ for every $\overline{u}_{1} \in \overline{\cl{U}}_{1}$ 
and $L_{\overline{u}_{2}}(\zeta) \in \cl{U}_{1}$ for every $\overline{u}_{2} \in \overline{\cl{U}}_{2}$
if and only if
$\zeta \in \cl{U}_{1}\hspace{-0.1cm}\Leftrightarrow\hspace{-0.05cm}\cl{U}_{2}$. 
\end{itemize}
\end{lemma}

\begin{proof}
(i) For the forward implication, write $\zeta = \zeta
_{1} + \zeta_{2}$ where 
$\zeta_{1} \in \cl{U}_{1}\otimes(\overline{\bb{C}}^{X_{2}}\otimes \bb{C}^{Y_{2}})$ and 
$\zeta_{2} \in \cl{U}_{1}^{\perp}\otimes (\overline{\bb{C}}^{X_{2}}\otimes \bb{C}^{Y_{2}})$. 
If $\overline{u}_{1} \in \overline{\cl{U}}_{1}$ then 
\begin{gather*}
	L_{\overline{u}_{1}}(\zeta)
	 = L_{\overline{u}_{1}}(\zeta_{1}) + L_{\overline{u}_{1}}(\zeta_{2}) 
	 = L_{\overline{u}_{1}}(\zeta_{1}) \in \cl{U}_{2},
\end{gather*}
implying $\zeta_{1} \in \cl{U}_{1}\otimes \cl{U}_{2}$ and hence
$\zeta \in \cl{U}_{1} \hspace{-0.1cm}\Rightarrow\hspace{-0.05cm}\cl{U}_{2}$.
The converse implication is immediate from the definition of the space 
$\cl{U}_{1}\hspace{-0.1cm}\Rightarrow\hspace{-0.05cm}\cl{U}_{2}$.

(ii) 
For the forward implication, decompose $\zeta$ as in (i) and let $u_{2} \in \cl{U}_{2}$. 
%then $L_{u_{2}}(\zeta_{2}) = 0$; 
Since $L_{\overline{u}_{2}}(\zeta_{1}) \in \cl{U}_{1}$ and 
$L_{\overline{u}_{2}}(\zeta_{2}) \in \cl{U}_{1}^{\perp}$, 
we have that $L_{\overline{u}_{2}}(\zeta_{2}) = 0$.
It follows that $\zeta_{2} \in \cl{U}_{1}^{\perp}\otimes \cl{U}_{2}^{\perp}$, 
and thus $\zeta \in \cl{U}_{1}\hspace{-0.1cm}\Leftrightarrow\hspace{-0.05cm}\cl{U}_{2}$.
As in (i), the converse implication is immediate from the definition of 
$\cl{U}_{1}\hspace{-0.1cm}\Leftrightarrow\hspace{-0.05cm}\cl{U}_{2}$.
\end{proof}

\begin{proposition}\label{affine_sim}
Let 
$\cl{U}_{1} \subseteq \bb{C}^{X_{1}}\otimes \overline{\bb{C}}^{Y_{1}}$
and 
$\cl{U}_{2} \subseteq \overline{\bb{C}}^{X_{2}}\otimes \bb{C}^{Y_{2}}$ be quantum hypergraphs,
and $\Gamma$ be a QNS correlation over $(X_{2}, Y_{1}, X_{1}, Y_{2})$. 
Then 
\begin{itemize}
\item[(i)]
$\cl{U}_{1}\leadsto_{\rm ns} \cl{U}_{2}$ via $\Gamma$ if and only if the map 
$\cl{E} \mapsto \Gamma[\cl{E}]$ restricts to a well-defined affine map from $\cl{Q}(\overline{\cl{U}}_{1})$ into $\cl{Q}(\cl{U}_{2})$;
\item[(ii)]
$\cl{U}_{1}\rightarrow_{\rm ns} \cl{U}_{2}$ via $\Gamma$ if and only if $\cl{E} \mapsto \Gamma[\cl{E}]$ (resp. $\cl{F} \mapsto \Gamma[\cl{F}]$) restricts to a well-defined affine map from $\cl{Q}(\overline{\cl{U}}_{1})$ into $\cl{Q}(\cl{U}_{2})$ (resp. $\cl{Q}(\overline{\cl{U}}_{1}^{\perp})$ into $\cl{Q}(\cl{U}_{2}^{\perp})$). 
\end{itemize}
\end{proposition}

\begin{proof}
(i) 
Let $\cl{E}: M_{X_{1}}\rightarrow M_{Y_{1}}$ be a quantum channel with $\cl{E} \in \cl{Q}(\overline{\cl{U}}_{1})$.
We first note that, by the definition (\ref{eq_defsim}),
$$(\Gamma_1 + \Gamma_2)[\cl E] = \Gamma_1[\cl E] +  \Gamma_2[\cl E],$$
whenever $\Gamma_i : M_{X_2Y_1}\to M_{X_1Y_2}$ are linear maps. 
By Proposition \ref{p_twisted-Choi}, in order to establish the forward implication, 
it suffices to show that, whenever $\Gamma : M_{X_2Y_1}\to M_{X_1Y_2}$ is a completely positive map, 
given by $\Gamma(T) = NTN^*$ for some operator $N : \bb{C}^{X_2Y_1}\to \bb{C}^{X_1Y_2}$, 
such that $\Gamma \in \cl Q(\cl{U}_{1}\hspace{-0.1cm}\rightarrow\hspace{-0.05cm}\cl{U}_{2})$, 
we have that 
$\tilde{\cl K}_{\Gamma[\cl E]}\subseteq \tilde{\cl{U}}_{2}$. 
By additivity, we may further assume that $\cl E$ is a completely positive map, given by 
$\cl E(S) = MSM^*$, for some operator $M : \bb{C}^{X_1}\to \bb{C}^{Y_1}$. 

If $N = \sum_{p = 1}^r A_p\otimes B_p$, where $A_p : \bb{C}^{X_2}\to \bb{C}^{X_1}$ and 
$B_p : \bb{C}^{Y_1}\to \bb{C}^{Y_2}$, let 
$N[M] = \sum_{p=1}^r B_p M A_p$; thus, $N[M] : \bb{C}^{X_2}\to \bb{C}^{Y_2}$ is an 
(easily seen to be well-defined) linear map.
Writing $\Gamma_{p,q} : M_{X_2Y_1}\to M_{X_1Y_2}$ for the map, given by 
$\Gamma_{p,q}(T) = (A_p\otimes B_p)T(A_q\otimes B_q)^*$, we have that,
if $S\in M_{X_2}$, then 
$$\Gamma[\cl E] (S) = \sum_{p,q=1}^r \Gamma_{p,q} [\cl E] (S) 
= \sum_{p,q=1}^r (B_pMA_p) S (B_qMA_q)^* = N[M] S N[M]^*.$$
It thus suffices to show that 
\begin{equation}\label{eq_NM}
N[M]\in \tilde{\cl U}_2.
\end{equation}

We note that $\theta^{-1}(N)\in \overline{\bb{C}}^{X_2}\otimes \overline{\bb{C}}^{Y_1}
\otimes \bb{C}^{X_1}\otimes \bb{C}^{Y_2}$, so that 
$(\sigma\circ \theta^{-1})(N)\in \bb{C}^{X_1}\otimes \overline{\bb{C}}^{Y_1}
\otimes \overline{\bb{C}}^{X_2} \otimes \bb{C}^{Y_2}$. 
On the other hand, 
$\theta^{-1}(M)\in \overline{\bb{C}}^{X_1}\otimes \bb{C}^{Y_1}$;
thus, $\theta^{-1}(M)$ can be viewed as a linear functional on $\bb{C}^{X_1}\otimes \overline{\bb{C}}^{Y_1}$.
%(\ref{eq_NM}) follows from Lemma \ref{slice_inclusion}, along with the following 
We prove the following:

\medskip

\noindent {\it Claim. } $\theta^{-1}(N[M]) = L_{\theta^{-1}(M)}((\sigma\circ \theta^{-1})(N))$.

\smallskip

\noindent {\it Proof of Claim. } By linearity, we may assume that 
$\theta^{-1}(N) = \overline{\xi}_2\otimes\overline{\eta}_1 \otimes \xi_1\otimes \eta_2$, where 
$\xi_i\in \bb{C}^{X_i}$ and $\eta_i \in \bb{C}^{Y_i}$, $i = 1,2$, and 
$\theta^{-1}(M) = \overline{\xi} \otimes\eta$, where $\xi \in \bb{C}^{X_1}$ and 
$\eta\in \bb{C}^{Y_1}$. 
We have that 
$$(\sigma\circ \theta^{-1})(N) = \xi_1 \otimes\overline{\eta}_1 \otimes \overline{\xi}_2\otimes \eta_2,$$
that 
$N = (\xi_1\xi_2^*) \otimes (\eta_2\eta_1^*)$ and 
$M = \eta\xi^*$, and hence that 
$$N[M] = (\eta_2\eta_1^*) (\eta\xi^*) (\xi_1\xi_2^*) = \langle \xi_1,\xi\rangle \langle \eta,\eta_1\rangle \eta_2\xi_2^*.
$$
Thus, 
$$\theta^{-1}(N[M]) = \langle \xi_1,\xi\rangle \langle \eta,\eta_1\rangle \overline{\xi}_2 \otimes \eta_2
= L_{\overline{\xi} \otimes\eta}(\xi_1 \otimes\overline{\eta}_1 \otimes \overline{\xi}_2\otimes \eta_2),
$$
and the claim follows. 

\smallskip

We now show (\ref{eq_NM}). By assumption, 
$(\sigma\circ \theta^{-1})(N) \in \cl{U}_{1}\hspace{-0.1cm}\Rightarrow\hspace{-0.05cm}\cl{U}_{2}$ and 
$\theta^{-1}(M)\in \overline{\cl U}_1$. 
By the Claim and Lemma \ref{slice_inclusion}, $\theta^{-1}(N[M])\in \cl U_2$, and hence 
$N[M]\in \tilde{\cl U}_2$.

Conversely, suppose that the map 
$\cl{E} \mapsto \Gamma[\cl{E}]$ restricts to a well-defined affine map from 
$\cl{Q}(\overline{\cl{U}}_{1})$ into $\cl{Q}(\cl{U}_{2})$.
By Proposition \ref{p_twisted-Choi}, it suffices to assume that $\Gamma$ has the form 
$\Gamma(T) = NTN^*$ for some operator $N : \bb{C}^{X_2Y_1}\to \bb{C}^{X_1Y_2}$. 
By the proof of (i), this implies that 
$N[M]\in \tilde{\cl U}_2$ whenever $\theta^{-1}(M)\in \overline{\cl U}_1$. 
Therefore, by the Claim, $L_{\theta^{-1}(M)}((\sigma\circ \theta^{-1}(N)) \in \mathcal{U}_{2}$ whenever $\theta^{-1}(M) \in \overline{\cl U}_{1}$; now Lemma \ref{slice_inclusion} shows that 
$(\sigma\circ \theta^{-1})(N) \in \cl{U}_{1}\hspace{-0.1cm}\Rightarrow\hspace{-0.05cm}\cl{U}_{2}$, 
and the proof of part (i) is complete. 

(ii) follows in a similar fashion to (i) and the detailed proof is omitted. 
\end{proof}

For a hypergraph $E\subseteq X\times Y$, we let 
$$\cl{C}(E) = \{\cl N : \cl D_X\to \cl D_Y \ : \ \mbox{channel s.t. }
\langle\cl N(\epsilon_{x,x}),\epsilon_{y,y}\rangle = 0 \mbox{ if } (x,y)\not\in E\}.$$
Recall that for full hypergraphs 
$E_{i} \subseteq X_{i}\times Y_{i}, i = 1, 2$, and NS correlation 
$\Gamma$ over $(X_{2}, Y_{1}, X_{1}, Y_{2})$, we have that 
$E_{1}\leadsto_{\rm ns} E_{2}$ via $\Gamma$ if and only if 
$\cl{E} \mapsto \Gamma[\cl{E}]$ restricts to a well-defined affine map of $\cl{C}(E_{1})$ to $\cl{C}(E_{2})$ 
(see \cite[Proposition 3.2]{ght}). 
We note that as a direct result of Proposition \ref{affine_sim}, we have the following strengthened statement.
%\end{remark}

\begin{corollary}
Let $E_{i} \subseteq X_{i}\times Y_{i}, i = 1, 2$ be full hypergraphs, and $\Gamma$ a NS correlation over $(X_{2}, Y_{1}, X_{1}, Y_{2})$. Then $E_{1}\rightarrow_{\rm ns} E_{2}$ via $\Gamma$ if and only if the map $\cl{E} \mapsto \Gamma[\cl{E}]$ (resp. $\cl{F} \mapsto \Gamma[\cl{F}]$) restricts to a well-defined affine map from 
$\cl{C}(E_{1})$ into $\cl{C}(E_{2})$ (resp. $\cl{C}(E_{1}^{\rm c})$ into $\cl{C}(E_{2}^{\rm c})$).
\end{corollary}

%In the rest of this section, we link quantum quasi-homomorphisms 
%of a certain given type to relations on the set of quantum hypergraphs of the same type. 
%We will need a modification of the notions of reflexivity and transitivity. 
Let $\cl R$ be a relation on the set of all quantum hypergraphs. We say that $\cl R$ is
\begin{itemize}
	\item[(i)] \textit{pseudo-reflexive} if $(\overline{\cl{U}}, \cl{U}) \in \cl R$ for every 
	quantum hypergraph $\cl U$;
	\item[(ii)] \textit{pseudo-transitive} if $(\overline{\cl{U}}_{1}, \cl{U}_{2}) \in \cl R$ and 
	$(\overline{\cl{U}}_{2}, \cl{U}_{3}) \in \cl R$ $\Longrightarrow$
 $(\overline{\cl{U}}_{1}, \cl{U}_{3}) \in \cl R$.
\end{itemize}
A relation which is pseudo-reflexive and pseudo-transitive will be called a \textit{pseudo-quasi-order}.

\begin{theorem}\label{quasi_hom_pseudo_order}
Let ${\rm t} \in \{\rm loc, q, qa, qc, ns\}$. Then the 
relations $\leadsto_{\rm t}$ and $\rightarrow_{\rm t}$
are pseudo-quasi-orders on the set of all quantum hypergraphs. 
\end{theorem}

\begin{proof}
Given a quantum hypergraph $\cl{U} \subseteq \overline{\bb{C}}^{X}\otimes \bb{C}^{Y}$, 
the identity channel ${\rm id} : M_{XY}\to M_{XY}$ fits 
$\overline{\cl{U}}\hspace{-0.1cm}\leftrightarrow \hspace{-0.05cm}\cl{U}$; 
since ${\rm id}\in \cl Q_{\rm loc}$, the inclusions 
(\ref{eq_Qinc}) imply that $\cl U_1\leadsto_{\rm t}\cl U_2$. 

Let $\cl{U}_{i} \subseteq \overline{\bb{C}}^{X_{i}}\otimes \bb{C}^{Y_{i}}$, $i = 1, 2, 3$ be quantum hypergraphs, 
and assume that 
$\overline{\cl{U}}_{1} \leadsto_{\rm t} \cl{U}_{2}$ via $\Gamma_{1}$, 
and $\overline{\cl{U}}_{2} \leadsto \cl{U}_{3}$ via $\Gamma_{2}$. 
By Theorem \ref{comp_thm}, $\Gamma_{2}\ast \Gamma_{1} \in \cl{Q}_{\rm t}$. 
We claim that $\overline{\cl{U}}_{1} \leadsto_{\rm t} \cl{U}_{3}$ via $\Gamma_{2}\ast \Gamma_{1}$; 
indeed, as $\Gamma_{1}$ fits $\overline{\cl{U}}_{1}\rightarrow \cl{U}_{2}$, and $\Gamma_{2}$ fits $\overline{\cl{U}}_{2}\rightarrow \cl{U}_{3}$, Proposition \ref{affine_sim} implies that 
$\cl{E} \mapsto \Gamma_{1}[\cl{E}]$ (resp. $\cl{F} \mapsto \Gamma_{2}[\cl{F}]$) restricts to a well-defined map 
from $\cl{Q}(\cl{U}_{1})$ into $\cl{Q}(\cl{U}_{2})$ (resp. from $\cl{Q}(\cl{U}_{2})$ into $\cl{Q}(\cl{U}_{3})$). 
If $\cl{E} \in \cl{Q}(\cl{U}_{1})$ then, by Theorem \ref{comp_thm}, 
$(\Gamma_{2}
\ast\Gamma_{1})[\cl{E}] = \Gamma_{2}[\Gamma_{1}[\cl{E}]] \in \cl{Q}(\cl{U}_{3})$; hence, 
by Proposition \ref{affine_sim}, 
$\overline{\cl{U}}_{1}\leadsto_{\rm t} \cl{U}_{3}$ via $\Gamma_{2}\ast \Gamma_{1}$.

Finally, assume that 
$\overline{\cl{U}}_{1}\rightarrow_{\rm t} \cl{U}_{2}$ via $\Gamma_{1}$, and $\overline{\cl{U}}_{2}\rightarrow_{\rm t} \cl{U}_{3}$ via $\Gamma_{2}$. 
By Proposition \ref{affine_sim}, 
$\cl{E} \mapsto \Gamma_{1}[\cl{E}]$ and $\cl{F} \mapsto \Gamma_{1}[\cl{F}]$ (resp. from $\cl{E}' \mapsto \Gamma_{2}[\cl{E}']$ and $\cl{F}' \mapsto \Gamma_{2}[\cl{F}']$) restrict to well-defined maps from 
$\cl{Q}(\cl{U}_{1})$ into $\cl{Q}(\cl{U}_{2})$,  and $\cl{Q}(\cl{U}_{1}^{\perp})$ into $\cl{Q}(\cl{U}_{2}^{\perp})$ (resp. 
 from $\cl{Q}(\cl{U}_{2})$ into $\cl{Q}(\cl{U}_{3})$, and 
from $\cl{Q}(\cl{U}_{2}^{\perp})$ into $\cl{Q}(\cl{U}_{3}^{\perp})$). 
As in the previous paragraph, another application of Proposition \ref{affine_sim} shows that 
$\overline{\cl{U}}_{1}\rightarrow_{\rm t} \cl{U}_{3}$ via $\Gamma_{2}\ast \Gamma_{1}$. 
\end{proof}

\noindent \indent Recall that for a finite set $Z$, the classical ``diagonal" hypergraph $\Delta_{Z} \subseteq Z\times Z$ is defined by letting
\begin{gather*}
	\Delta_{Z} := \{(z, z): \; z \in Z\}.
\end{gather*}
\begin{proposition}
The implications
\begin{gather*}
	\cl{U}_{1} \leadsto_{\rm loc} \cl{U}_{2} \;\; \Rightarrow \;\; \cl{U}_{1} \leadsto_{\rm q} \cl{U}_{2} \;\;\;\; \text{ and } \;\;\;\; \cl{U}_{1} \leadsto_{\rm q} \cl{U}_{2} \;\; \Rightarrow \;\; \cl{U}_{1} \leadsto_{\rm ns} \cl{U}_{2}
\end{gather*}
\noindent are not reversible. 
\end{proposition}
\noindent \begin{proof}
Using \cite[Proposition 3.6]{ght}, there exist finite sets $X, Y, Z, Z'$ and classical hypergraphs $E, E' \subseteq X\times Y$ such that $E \leadsto_{\rm q} \Delta_{Z}$ and $E' \leadsto_{\rm ns} \Delta_{Z'}$ while $E \not \leadsto_{\rm loc} \Delta_{Z}$ and $E' \not \leadsto_{\rm q} \Delta_{Z'}$. Applying Proposition \ref{classical_to_quantum}, we have 
that $\cl{U}_{E} \leadsto_{\rm q} \cl{U}_{\Delta_{Z}}$ while $\cl{U}_{E} \not \leadsto_{\rm loc} \cl{U}_{\Delta_{Z}}$, 
and that $\cl{U}_{E'} \leadsto_{\rm ns} \cl{U}_{\Delta_{Z'}}$ while $\cl{U}_{E'} \not\leadsto_{\rm q} \cl{U}_{\Delta_{Z'}}$. 
\end{proof}

%%%%%%%%%%%%%%%%%%%%%%%%%%%%%%%%%%%%%%%%%%%%%%%%%
%%%%%%%%%%%%%%%%%%%%%%%%%%%%%%%%%%%%%%%%%%%%%%%%%

\section{Homomorphisms of local and no-signalling type}\label{s_Morita}

In what follows, we fix finite sets $X_i$ and $Y_i$, $i = 1,2$, and 
quantum hypergraphs $\cl{U}_{1} \subseteq \bb{C}^{X_{1}}\otimes \overline{\bb{C}}^{Y_{1}}$
and $\cl{U}_{2} \subseteq \overline{\bb{C}}^{X_{2}}\otimes \bb{C}^{Y_{2}}$;
we thus have that 
$$ \tilde{\cl{U}}_{1} \subseteq \cl{L}(\overline{\bb{C}}^{X_{1}}, \overline{\bb{C}}^{Y_{1}}) \text{ and } \tilde{\cl{U}}_{2} \subseteq \cl{L}(\bb{C}^{X_{2}}, \bb{C}^{Y_{2}}).$$
Set 
$ \hat{\cl{U}}_{1} := \{\overline{T}: \; T \in \tilde{\cl{U}}_{1}\},$
and note that 
$\hat{\cl{U}}_{1} \subseteq \cl{L}(\bb{C}^{Y_{1}}, \bb{C}^{X_{1}})$. Thus, 
$$ \hat{\cl{U}}_{1}^{*} := \{S^{*}: \; S \in \hat{\cl{U}}_{1}\} \subseteq \cl{L}(\bb{C}^{X_{1}}, \bb{C}^{Y_{1}}).$$

For finite dimensional Hilbert spaces $H$ and $K$, 
%\marginpar{\tiny IT. It seems we should be talking about column, and not row, isometries -- right?}
% GH. Correct- removed comment.
we will say that a linear subspace $\cl M\subseteq \cl L(H,K)$ \emph{contains a column isometry}
if there exist operators $A_i\in \cl M$, $i = 1,\dots,m$, such that $\sum_{i=1}^m A_i^* A_i = I$.

\begin{lemma}\label{qloc_quasi}
Let 
$\cl{U}_{1} \subseteq \bb{C}^{X_{1}}\otimes \overline{\bb{C}}^{Y_{1}}$
and $\cl{U}_{2} \subseteq \overline{\bb{C}}^{X_{2}}\otimes \bb{C}^{Y_{2}}$ be quantum hypergraphs. 
The following are equivalent:
\begin{itemize}
	\item[(i)] $\cl{U}_{1} \leadsto_{\rm loc} \cl{U}_{2}$;
	\item[(ii)] there exist subspaces $\cl{L} \subseteq \cl{L}(\bb{C}^{Y_{1}}, \bb{C}^{Y_{2}})$ 
	and $\cl{R} \subseteq \cl{L}(\bb{C}^{X_{2}}, \bb{C}^{X_{1}})$ containing column 
	isometries, such that 
	$\cl{L}\hspace{0.1cm}\hat{\cl{U}}_{1}^{*}\hspace{0.05cm}\cl{R} \subseteq \tilde{\cl{U}}_{2}$.
\end{itemize} 
\end{lemma}

\begin{proof}
We first establish the following auxiliary claim:

\medskip

\noindent {\it Claim. } If $u\in \bb{C}^{X_1}\otimes \overline{\bb{C}}^{Y_1}$, 
$A \in \cl L(\bb{C}^{X_2},\bb{C}^{X_1})$ and $B \in \cl L(\bb{C}^{Y_1},\bb{C}^{Y_2})$ then 
$$(\overline{A}\otimes B)(\overline{u}) = 
L_{\overline{u}}\left((\sigma\circ\theta^{-1})(A\otimes B)\right).$$

\smallskip

\noindent {\it Proof of Claim. } 
We may assume, by linearity, that 
$A = \xi_1\xi_2^*$, $B = \eta_2\eta_1^*$ and $u = \xi_0\otimes \overline{\eta}_0$, where 
$\xi_0, \xi_1\in \bb{C}^{X_1}$, $\xi_2\in \bb{C}^{X_2}$, 
$\eta_0,\eta_1\in \bb{C}^{Y_1}$ and $\eta_2\in \bb{C}^{Y_2}$. 
Thus, 
$A\otimes B = (\xi_1\otimes\eta_2)(\xi_2\otimes\eta_1)^*$ and so 
$$(\sigma\circ\theta^{-1})(A\otimes B) = \sigma(\overline{\xi}_2\otimes \overline{\eta}_1\otimes \xi_1\otimes\eta_2)
= \xi_1 \otimes \overline{\eta}_1\otimes \overline{\xi}_2\otimes\eta_2.$$
It follows that 
\begin{eqnarray*}
(\overline{A}\otimes B)(\overline{u}) & = & 
((\overline{\xi}_2\overline{\xi}_1^*)\otimes (\eta_2\eta_1^*))(\overline{u})
= 
(\overline{\xi}_2 \otimes \eta_2)(\overline{\xi}_1\otimes \eta_1)^*(\overline{\xi}_0\otimes \eta_0)\\
& = & 
\langle \xi_1,\xi_0\rangle \langle \eta_0,\eta_1\rangle \overline{\xi}_2\otimes\eta_2
= L_{\overline{u}}\left((\sigma\circ\theta^{-1})(A\otimes B)\right) .
\end{eqnarray*}

\smallskip

(i)$\Rightarrow$(ii) 
Assume $\Gamma: M_{X_{2}Y_{1}}\rightarrow M_{X_{1}Y_{2}}$ is a channel in $\cl{Q}_{\rm loc}$ which fits $\cl{U}_{1}\rightarrow \cl{U}_{2}$. 
By Proposition \ref{p_twisted-Choi}, we may assume that 
$\Gamma = \Phi \otimes \Psi$, for some quantum channels 
$\Phi: M_{X_{2}}\rightarrow M_{X_{1}}$ and $\Psi: M_{Y_{1}}\rightarrow M_{Y_{2}}$. 
Write $\Phi(S) = \sum_{i=1}^{n}A_{i}SA_{i}^{*}$ and $\Psi(R) = \sum_{j=1}^{m}B_{j}RB_{j}^{*}$
in their Kraus representations; thus, 
$$\Gamma(T) = \sum_{i=1}^{n} \sum_{j=1}^{m} (A_{i}\otimes B_j) T (A_{i}\otimes B_j)^{*}, \ \ \ 
T\in M_{X_2Y_1}.$$
By assumption, 
$$\zeta_{i,j} := (\sigma\circ\theta^{-1})(A_i\otimes B_j) \in 
\hspace{0.05cm} \cl U_1\hspace{-0.1cm}\Rightarrow \hspace{-0.05cm} \cl U_2, 
\ \ \ i\in [n], j\in [m].$$
By the Claim and Lemma \ref{slice_inclusion},  
$(\overline{A}_i\otimes B_j)(\overline{u})\in \cl U_2$ for all $i\in [n]$, all $j\in [m]$, 
and all $\overline{u}\in \overline{\mathcal{U}}_{1}$. 
By Lemma \ref{l_bars}, $B_j\overline{\theta(u)}^* A_i\in \tilde{\cl U}_2$ for all $i\in [n]$, 
all $j\in [m]$ and all $\overline{u}\in \overline{\mathcal{U}}_{1}$. 
Letting $\cl{L} = \tilde{\cl{K}}_{\Psi}$ and $\cl{R} = \tilde{\cl{K}}_{\Phi}$, we thus have that 
$\cl L$ (resp. $\cl R$) contains the column isometry $(B_j)_{j=1}^m$ (resp. $(A_i)_{i=1}^n$) and 
$\cl L\hat{\cl{U}}_{1}^{*}\cl R\subseteq \tilde{\cl U}_2$.

(ii)$\Rightarrow$ (i) 
Let $\cl{R}$ and $\cl{L}$ be subspaces, containing column isometries 
$\{A_{i}\}_{i=1}^{n}$ and $\{B_{j}\}_{j=1}^{m}$, respectively, such that
$\cl{L}\hspace{0.1cm}\hat{\cl{U}}_{1}^{*}\hspace{0.05cm}\cl{R} \subseteq \tilde{\cl{U}}_{2}$. 
Let
$\Phi : M_{X_{2}}\rightarrow M_{X_{1}}$, $\Psi: M_{Y_{1}}\rightarrow M_{Y_{2}}$ be the quantum channels, given by 
$\Phi(S) = \sum_{i=1}^{n}A_{i}SA_{i}^{*}$ and $\Psi(R) = \sum_{j=1}^{m}B_{j}RB_{j}^{*}$. 
We have that $\Gamma := \Phi\otimes \Psi$ is a quantum channel, belonging to 
$\cl{Q}_{\rm loc}$. 
Reversing the steps from the previous paragraph, we see that 
$L_{\overline{u}}((\sigma\circ\theta^{-1})(A_i\otimes B_j))\in \cl U_2$ for all 
$\overline{u}\in \overline{\mathcal{U}}_{1}$, all $i\in [n]$ and all 
$j\in [m]$. 
By Lemma \ref{slice_inclusion}, 
$$(\sigma\circ\theta^{-1})(A_i\otimes B_j)\in 
\hspace{0.05cm} \cl{U}_{1}\hspace{-0.1cm}\Rightarrow\hspace{-0.05cm}\cl{U}_{2}, 
\ \ \ i\in [n], j\in [m],$$
and hence 
$\cl{K}_{\Gamma} \subseteq \cl{U}_{1}\hspace{-0.1cm}\Rightarrow\hspace{-0.05cm}\cl{U}_{2}$, 
implying that $\cl{U}_{1}\leadsto_{\rm loc} \cl{U}_{2}$ via $\Gamma$.
\end{proof}

\begin{lemma}\label{kernel_prf}
If $\cl{L} \subseteq \cl{L}(H, K)$ is a subspace of linear operators between finite
dimensional Hilbert spaces $H, K$, such that 
$\cap_{T \in \cl{L}}\ker T = \{0\}$, then there exist operators $T_{1}, \hdots, T_{n} \in \cl{L}$ such that 
$\cap_{i=1}^{n}\ker T_{i} = \{0\}$.
\end{lemma}

\begin{proof}
Set $\ker \cl{L} := \cap_{T \in \cl{L}}\ker T$; thus, $\ker \cl{L} = \{0\}$. 
Let $T_{1} \in \cl{L}$. If $\ker T_{1} = \{0\}$, we are done; otherwise, let 
$T_2\in \cl L$ be such that $\ker T_1\not\subseteq\ker T_2$ (note that the existence of 
such $T_2$ is guaranteed by the fact that $\ker \cl{L} = \{0\}$). 
If $\ker T_1\cap \ker T_2 = \{0\}$, we are done; otherwise, once again by the assumption in the statement, 
we can choose $T_3\in \cl L$ such that $\ker T_1\cap \ker T_2\not\subseteq\ker T_3$. 
Continuing inductively, the finite dimensionality of $H$ guarantees that the process will terminate. 
\end{proof}

Let $H$ and $K$ be Hilbert spaces. 
Recall \cite{harris} that a \emph{ternary ring of operators (TRO)} is a subspace $\cl M\subseteq \cl B(H,K)$ such that 
$$S, T, R\in \cl M \ \Longrightarrow \ ST^*R\in \cl M.$$
A TRO $\cl M\subseteq \cl B(H,K)$ will be called 
\emph{left non-degenerate} 
(resp. \emph{right non-degenerate}) if $\overline{{\rm span}(\cl M^* K)} = H$
(resp. $\overline{{\rm span}(\cl M H)} = K$).
We call $\cl M$ \emph{non-degenerate}  if it is both left and right non-degenerate. 
Let $\cl S_1\subseteq \cl L(\bb{C}^{X_1},\bb{C}^{Y_1})$ and 
$\cl S_2\subseteq \cl L(\bb{C}^{X_2},\bb{C}^{Y_2})$ be operator spaces. 
Recall \cite{elef, ept} that  $\cl S_1$ and $\cl S_2$ 
are called 
\emph{TRO equivalent} (denoted $\cl S_1\sim_{\rm TRO} \cl S_2$) 
if there exist non-degenerate TRO's 
$\cl{L} \subseteq \cl{L}(\bb{C}^{Y_{1}}, \bb{C}^{Y_{2}})$ and 
	$\cl{R} \subseteq \cl{L}(\bb{C}^{X_{2}}, \bb{C}^{X_{1}})$ such that 

\begin{equation}\label{eq_TROeq}
\cl{L}\hspace{0.1cm}\cl S_1 \hspace{0.05cm}\cl{R} \subseteq \cl S_2
\ \ \mbox{ and } \ \  \cl{L}^{*}\cl S_2 \cl{R}^{*} \subseteq \cl S_1.
\end{equation}
We will say that $\cl S_1$ is \emph{TRO left homomorphic} to $\cl S_2$ 
(denoted $\cl S_1\to_{\rm TRO} \cl S_2$)
if there exists left non-degenerate TRO's $\cl L$ and $\cl R$ for which 
(\ref{eq_TROeq}) holds true.

\begin{theorem}\label{qloc_hom}
$\cl{U}_{1} \subseteq \bb{C}^{X_{1}}\otimes \overline{\bb{C}}^{Y_{1}}$
and $\cl{U}_{2} \subseteq \overline{\bb{C}}^{X_{2}}\otimes \bb{C}^{Y_{2}}$ be quantum hypergraphs. 
The following are equivalent:
\begin{itemize}
	\item[(i)] $\cl{U}_{1} \rightarrow_{\rm loc} \cl{U}_{2}$;
	\item[(ii)] $\hat{\cl{U}}_{1}^{*}\to_{\rm TRO} \tilde{\cl{U}}_{2}$.
\end{itemize}
\end{theorem}

\begin{proof}
(i)$\Rightarrow$(ii) 
Assume $\Gamma: M_{X_{2}Y_{1}}\rightarrow M_{X_{1}Y_{2}}$ is a channel in $\cl{Q}_{\rm loc}$ which fits 
$\cl{U}_{1}\hspace{-0.1cm}\leftrightarrow\hspace{-0.05cm} \cl{U}_{2}$; 
as in the proof of Lemma \ref{qloc_quasi}, 
assume that $\Gamma = \Phi\otimes \Psi$, for some quantum channels 
$\Phi: M_{X_{2}}\rightarrow M_{X_{1}}$ and $\Psi: M_{Y_{1}}\rightarrow M_{Y_{2}}$. 
Write $\Phi(S) = \sum_{i=1}^{n}A_{i}SA_{i}^{*}$ and $\Psi(R) = \sum_{j=1}^{m}B_{j}RB_{j}^{*}$
in their Kraus representations, and let 
$\cl{L}_0 = \tilde{\cl{K}}_{\Psi}$ and $\cl{R}_0 = \tilde{\cl{K}}_{\Phi}$. 
By Lemma \ref{qloc_quasi}, 
\begin{equation}\label{eq_forTRO1}
\cl{L}_0 \hspace{0.1cm}\hat{\cl{U}}_{1}^{*}\hspace{0.05cm}\cl{R}_0 \subseteq \tilde{\cl{U}}_{2}.
\end{equation}

By Lemma \ref{l_bars}, 
\begin{equation}\label{eq_stars0}
\theta((\overline{A}_i^*\otimes B_j^*)(v)) = B_j^*\theta(v)A_i^*, \ \ 
v\in \overline{\bb{C}}^{X_2}\otimes \bb{C}^{Y_2}, i\in [n], j\in [m].
\end{equation}
Note that 
$(\overline{A}_i^*\otimes B_j^*)(v)\in \overline{\bb{C}}^{X_1}\otimes \bb{C}^{Y_1}$. 
By the Claim in the proof of Lemma \ref{qloc_quasi}, 
\begin{equation}\label{eq_stars}
(\overline{A}_i^*\otimes B_j^*)(v) = 
L_{v}\left((\sigma\circ\theta^{-1})(A_i^*\otimes B_j^*)\right), \ \ \ \ \ \ 
i\in [n], j\in [m].
\end{equation}
We claim that, for any subspace $\mathcal{U} \subseteq \overline{\bb{C}}^{X}\otimes \bb{C}^{Y}$ we have 
\begin{equation}\label{eq_for_equiv_btwn_hat}
	\tilde{\mathcal{U}} = \hat{(\overline{\mathcal{U}})}^{*}.
\end{equation}
\noindent Indeed, using  Lemma \ref{l_bars} and the fact that $\bb{C}^{X}, \bb{C}^{Y}$ are finite dimensional,  
for $u \in \mathcal{U}$ we have
$ \theta(u) = \theta(\overline{\overline{u}}) = \overline{\theta(\overline{u})}^{*}$.
As left-most element in the last string of equalities lies (by definition) in $\tilde{\mathcal{U}}$, while the right-most 
lies in $\hat{(\overline{\mathcal{U}})}^{*}$, identity (\ref{eq_for_equiv_btwn_hat}) is established. 

We next show that
\begin{equation}\label{eq_symst}
(\sigma \circ \theta^{-1})(A_{i}^{*}\otimes B_{j}^{*}) \in \overline{\mathcal{U}}_{2} \Rightarrow \overline{\mathcal{U}}_{1}, 
\ \ \ i\in [n], j\in [m].
\end{equation}
To this end, fix $i \in [n]$ and $j \in [m]$, and assume without loss of generality
(as in the proof of Lemma \ref{qloc_quasi}) that $A_{i} = \xi_{1}\xi_{2}^{*}$ and $B_{j} = \eta_{2}\eta_{1}^{*}$, 
where $\xi_{1} \in \bb{C}^{X_{1}}, \xi_{2} \in \bb{C}^{X_{2}}, \eta_{1} \in \bb{C}^{Y_{1}}$ and $\eta_{2} \in \bb{C}^{Y_{2}}$. Then $A_{i}^{*}\otimes B_{j}^{*} = (\xi_{2}\otimes \eta_{1})(\xi_{1}\otimes \eta_{2})^{*}$, and so
$$	(\sigma\circ \theta^{-1})(A_{i}^{*}\otimes B_{j}^{*}) = 
	\sigma(\overline{\xi}_{1}\otimes \overline{\eta}_{2}\otimes \xi_{2}\otimes \eta_{1}) =  
 \xi_{2}\otimes \overline{\eta}_{2}\otimes \overline{\xi}_{1}\otimes \eta_{1}.$$
Identity (\ref{eq_symst}) now follows 
by an application of Lemma \ref{slice_inclusion}, using an argument similar to that of Lemma \ref{qloc_quasi}.
In conjunction with (\ref{eq_forTRO1}) and (\ref{eq_for_equiv_btwn_hat}), identity (\ref{eq_symst})
implies

\begin{equation}\label{eq_forTRO2}
\cl{L}_0^* \hspace{0.1cm} \tilde{\cl{U}}_{2} \hspace{0.05cm}\cl{R}_0^* \subseteq \hat{\cl{U}}_{1}^{*}.
\end{equation}

Let 
$\cl{L}$ (resp. $\cl R$) be the TRO, generated by $\cl L_0$ (resp. $\cl R_0$). 
Using (\ref{eq_forTRO1}) and (\ref{eq_forTRO2}), it is straightforward to see
that 
$\cl{L}\hspace{0.1cm}\hat{\cl{U}}_{1}^{*}\hspace{0.05cm}\cl{R} \subseteq \tilde{\cl{U}}_{2}$ and 
$\cl{L}^{*}\tilde{\cl{U}}_{2} \cl{R}^{*} \subseteq \hat{\cl{U}}_{1}^{*}$.

(ii)$\Rightarrow$(i) 
Let $\cl{L}$ and $\cl{R}$ be TRO's satisfying the conditions of (ii). 
Since $\cl{L}^{*}\bb{C}^{Y_{2}} = \bb{C}^{Y_{1}}$ and $\cl{R}^{*}\bb{C}^{X_{1}} = \bb{C}^{X_{2}}$, we have that 
$$\cap_{L \in \cl{L}}\ker L = \{0\}  \ \ \mbox{ and } \ \  
\cap_{R \in \cl{R}} \ker R = \{0\}.$$
By Lemma \ref{kernel_prf},  there exist 
operators $A_{1}, \hdots, A_{n} \in \tilde{\cl{R}}$ (resp. $B_{1}, \hdots, B_{m} \in \tilde{\cl{L}}$) such that 
$$ \cap_{i=1}^{n}\ker A_{i} = \{0\} \ \ \mbox{ and } \ \  \cap_{j=1}^{m}\ker B_{j} = \{0\}.$$
We have that the operators
$K_{1} := \sum\limits_{i=1}^{n}A_{i}^{*}A_{i}$ and $K_{2} := \sum\limits_{j=1}^{m}B_{j}^{*}B_{j}$ are invertible.
By TRO functional calculus (see \cite[Section 3]{harris}),
$\tilde{A}_{i} := A_{i}K_{1}^{-1/2} \in \tilde{\cl{R}}$, $i\in [n]$  
(resp. $\tilde{B}_{j} := B_{j}K_{2}^{-1/2} \in \tilde{\cl{L}}$, $j \in [m]$). 
Let $\Phi : M_{X_{2}}\rightarrow M_{X_{1}}$ and $\Psi : M_{Y_{1}}\rightarrow M_{Y_{2}}$ 
be the quantum channels, given by 
$\Phi(S) = \sum_{i=1}^{n}\tilde{A}_{i}S\tilde{A}_{i}^{*}$ and $\Psi(R) = \sum_{j=1}^{m}\tilde{B}_{j}R\tilde{B}_{j}^{*}$. 
By Lemma \ref{qloc_quasi},
$\cl{U}_{1}\leadsto_{\rm loc} \cl{U}_{2}$ via $\Phi\otimes \Psi$.
On the other hand, (\ref{eq_symst}) implies that
$$(\sigma \circ \theta^{-1})(A_{i}\otimes B_{j}) \in 
(\cl{U}_{1}\otimes \cl{U}_{2})
+ \left(\overline{\bb{C}}^{X_{1}}\otimes \bb{C}^{Y_{1}}\right)\otimes \cl{U}_{2}^{\perp},$$
$i\in [n], j\in [m]$.
An application of Lemma \ref{slice_inclusion}
now shows that $\cl{U}_{1}\rightarrow_{\rm loc}\cl{U}_{2}$ via $\Phi\otimes \Psi$. 
\end{proof}

Let 
$\cl{U}_{1} \subseteq \bb{C}^{X_{1}}\otimes \overline{\bb{C}}^{Y_{1}}$
and 
$\cl{U}_{2} \subseteq \overline{\bb{C}}^{X_{2}}\otimes \bb{C}^{Y_{2}}$ be quantum hypergraphs.
For a correlation type ${\rm t}$, we will say that 
$\cl{U}_{1}$ is \emph{fully ${\rm t}$-homomorphic} to $\cl{U}_{2}$,
if there exists a quantum channel 
$\Gamma: M_{X_{2}Y_{1}}\rightarrow M_{X_{1}Y_{2}}$ with $\Gamma \in \cl{Q}_{\rm t}$
such that $\Gamma(I)$ is invertible and 
$\Gamma \in \cl{Q}(\cl{U}_{1}\hspace{-0.1cm}\leftrightarrow \hspace{-0.05cm}\cl{U}_{2})$.

\begin{theorem}\label{c_qloc_hom}
$\cl{U}_{1} \subseteq \bb{C}^{X_{1}}\otimes \overline{\bb{C}}^{Y_{1}}$
and $\cl{U}_{2} \subseteq \overline{\bb{C}}^{X_{2}}\otimes \bb{C}^{Y_{2}}$ be quantum hypergraphs. 
The following are equivalent:
\begin{itemize}
	\item[(i)] $\cl U_1$ is fully ${\rm loc}$-homomorphic to $\cl U_2$; 
	\item[(ii)] $\hat{\cl{U}}_{1}^{*} \sim_{\rm TRO} \tilde{\cl{U}}_{2}$.
	\end{itemize}
\end{theorem}

\begin{proof}
(i)$\Rightarrow$(ii) 
Let $\Gamma: M_{X_{2}Y_{1}}\rightarrow M_{X_{1}Y_{2}}$ be a local correlation 
that fits 
$\cl{U}_{1}\hspace{-0.1cm}\rightarrow\hspace{-0.05cm}\cl{U}_{2}$, 
such that $\Gamma(I)$ is invertible. 
Using the notation from the proof of Theorem \ref{qloc_hom}, 
we have that 
$\left(\sum_{i=1}^n A_iA_i^*\right)\otimes \left(\sum_{j=1}^m B_jB_j^*\right)$ is invertible, that is, 
the operators 
$\sum_{i=1}^n A_iA_i^*$ and $\sum_{j=1}^m B_jB_j^*$ are invertible. 
It follows that the TRO's $\cl L$ and $\cl R$ are non-degenerate, and hence, by the proof of 
Theorem \ref{qloc_hom},  $\hat{\cl{U}}_{1}^{*} \sim_{\rm TRO} \tilde{\cl{U}}_{2}$.

(ii)$\Rightarrow$(i)
Let  $\cl{L} \subseteq \cl{L}(\bb{C}^{Y_{1}}, \bb{C}^{Y_{2}})$ and 
$\cl{R} \subseteq \cl{L}(\bb{C}^{X_{2}}, \bb{C}^{X_{1}})$ be non-degenerate TRO's
for which (\ref{eq_TROeq}) holds true. 
Using Lemma \ref{kernel_prf}, choose 
operators $A_{1}, \hdots$, $A_{n_0}$, $A_{n_0+1}, \hdots, A_{n}  \in \tilde{\cl{R}}$ 
(resp. $B_{1}, \hdots, B_{m_0}, B_{m_0+1}, \hdots, B_{m} \in \tilde{\cl{L}}$) such that 
\begin{equation}\label{eq_withs}
\cap_{i=1}^{n_0}\ker A_{i} = \{0\}, \ \ \ \cap_{i=n_0+1}^{n}\ker A_{i}^* = \{0\},
\end{equation}
$$\cap_{j=1}^{m_0}\ker B_{j} = \{0\} \ \mbox{ and } \  \cap_{j=m_0+1}^{m}\ker B_{j}^* = \{0\}.$$
We thus have that the operators
$K_{1} := \sum\limits_{i=1}^{n}A_{i}^{*}A_{i}$,  and 
$K_{2} := \sum\limits_{j=1}^{m}B_{j}^{*}B_{j}$
are both invertible.
Let $\Phi_0 : M_{X_2}\to M_{X_1}$ and $\Psi_0 : M_{Y_1}\to M_{Y_2}$ be the 
completely positive maps, given by 
$\Phi_0(S) = \sum_{i=1}^{n} A_{i}S A_{i}^{*}$ and $\Psi_0(R) = \sum_{j=1}^{m} B_{j}R B_{j}^{*}$.
Following the proof of Theorem \ref{qloc_hom}, set 
$\tilde{A}_{i} := A_{i}K_{1}^{-1/2} \in \tilde{\cl{R}}$, $i\in [n]$  
(resp. $\tilde{B}_{j} := B_{j}K_{2}^{-1/2} \in \tilde{\cl{L}}$, $j \in [m]$), and let 
$\Phi : M_{X_2}\to M_{X_1}$ and $\Psi : M_{Y_1}\to M_{Y_2}$ be the 
quantum channels, given by 
$\Phi(S) = \sum_{i=1}^{n} \tilde{A}_{i}S \tilde{A}_{i}^{*}$ and 
$\Psi(R) = \sum_{j=1}^{m} \tilde{B}_{j}R \tilde{B}_{j}^{*}$. 
We have that 
$\Phi(I) = \sum_{i=1}^{n} A_{i} K_{1}^{-1}A_{i}^{*}$. 
The second of conditions (\ref{eq_withs}) implies that 
$\cap_{i=n_0+1}^{n}\ker K_1^{-1/2} A_{i}^* = \{0\}$, showing that $\Phi(I)$ is invertible. 
Similarly, $\Psi(I)$ is invertible, and hence $(\Phi\otimes\Psi)(I)$ is invertible.
The proof of Theorem \ref{qloc_hom} now implies that 
$\cl U_1$ is fully ${\rm loc}$-homomorphic to $\cl U_2$.
\end{proof}

We finish with a characterisation of no-signalling homomorphisms.

\begin{theorem}\label{qns_hom}
Let $\Gamma$ be a QNS correlation over $(X_{2}, Y_{1}, X_{1}, Y_{2})$. The following are equivalent:
\begin{itemize}
	\item[(i)] $\cl{U}_{1}\rightarrow_{\rm ns} \cl{U}_{2}$ via $\Gamma$;
	\item[(ii)] 
	%for Kraus space $\cl{K}_{\Gamma} \subseteq \cl{L}(\bb{C}^{X_{2}}\otimes \bb{C}^{Y_{1}}, \bb{C}^{X_{1}}\otimes \bb{C}^{Y_{2}})$, we have that 
	$N[U_{1}] \in \tilde{\cl{U}}_{2}$ and $N^{*}[U_{2}] \in \hat{\cl{U}}_{1}^{*}$ 
	whenever $N \in \cl{K}_{\Gamma}, U_{1} \in \hat{\cl{U}}_{1}^{*}$ and $U_{2} \in \tilde{\cl{U}}_{2}$. 
\end{itemize}
\end{theorem}

\begin{proof}
(i) $\Rightarrow$ (ii) As in Proposition \ref{affine_sim}, it is sufficient to assume $\Gamma: M_{X_{2}Y_{1}}\rightarrow M_{X_{1}Y_{2}}$ is a completely positive map with 
$\Gamma \in \cl{Q}(\cl{U}_{1}\hspace{-0.1cm}\leftrightarrow \hspace{-0.1cm}\cl{U}_{2})$. 
Let $N \in \cl{K}_{\Gamma}$, and write 
$N = \sum_{p=1}^{r}A_{p}\otimes B_{p}$, where $A_{p}: \bb{C}^{X_{2}}\rightarrow \bb{C}^{X_{1}}$ 
and $B_{p}: \bb{C}^{Y_{1}}\rightarrow \bb{C}^{Y_{2}}$.
Recall, from the proof of Proposition \ref{affine_sim} that, if 
$U_{1} \in \hat{\cl{U}}_{1}^{*}$, then
$N[U_{1}] := \sum_{p=1}^{r}B_{p}U_{1}A_{p}$; 
similarly, if 
$U_{2} \in \tilde{\cl{U}}_{2}$, then $N^{*}[U_{2}] = \sum_{p=1}^{r}B_{p}^{*}U_{2}A_{p}^{*}$. 
Using analogous arguments as given in the proofs of Proposition \ref{affine_sim} and 
Theorem \ref{qloc_hom}, we may conclude that $N[U_{1}] \in \tilde{\cl{U}}_{2}$ and 
$N^{*}[U_{2}] \in \hat{\cl{U}}_{1}^{*}$. 

(ii) $\Rightarrow$ (i) Assume that $N[U_{1}] \in \tilde{\cl{U}}_{2}$ and $N^{*}[U_{2}] \in \hat{\cl{U}}_{1}^{*}$ for each $N \in \cl{K}_{\Gamma}, U_{1} \in \hat{\cl{U}}_{1}^{*}$ and $U_{2} \in \tilde{\cl{U}}_{2}$. 
By the Claim in Proposition \ref{affine_sim}, $$ L_{\theta^{-1}(u_{1})}((\sigma\circ \theta^{-1}(N)) \in \cl{U}_{2}$$ whenever $\theta^{-1}(u_{1}) \in \overline{\cl{U}}_{1}$ and $$L_{\theta^{-1}(u_{2})}((\sigma \circ \theta^{-1}(N^{*})) \in \overline{\cl{U}}_{1}$$ whenever $\theta^{-1}(u_{2}) \in \cl{U}_{2}$. By Lemma \ref{slice_inclusion}, 
$(\sigma\circ \theta^{-1})(N) \in \cl{U}_{1} \hspace{-0.15cm}\Rightarrow \hspace{-0.15cm} \cl{U}_{2}$ and 
$(\sigma\circ \theta^{-1})(N^{*}) \in \overline{\cl{U}}_{2}  \hspace{-0.15cm}\Rightarrow  \hspace{-0.15cm} \overline{\cl{U}}_{1}$. Arguing as in the proof of Theorem \ref{qloc_hom}, this is equivalent to saying $(\sigma\circ \theta^{-1})(N) \in \cl{U}_{1} \hspace{-0.1cm}\Leftrightarrow \hspace{-0.1cm} \cl{U}_{2}$. 
As this holds for each $N \in \cl{K}_{\Gamma}$, we have that 
$\Gamma \in \cl{Q}(\cl{U}_{1} \hspace{-0.1cm}\leftrightarrow  \hspace{-0.1cm}\cl{U}_{2})$ and so 
$\cl{U}_{1}\rightarrow_{\rm ns} \cl{U}_{2}$ via $\Gamma$. 
\end{proof}

%%%%%%%%%%%%%%%%%%%%%%%%%%%%%%%%%%%%%%%%%%%%%%%%%
%%%%%%%%%%%%%%%%%%%%%%%%%%%%%%%%%%%%%%%%%%%%%%%%%
%%%%%%%%%%%%%%%%%%%%%%%%%%%%%%%%%%%%%%%%%%%%%%%%%

\end{document}